\documentclass{m2an}
\usepackage{mathtools}

\usepackage{lipsum}
\usepackage{amsfonts}
\usepackage{amssymb}
\usepackage{graphicx}
\usepackage{epstopdf}
\usepackage{algorithmic}
\usepackage{svg}
\usepackage{hyperref}
\usepackage{cleveref}
\usepackage{subcaption}
\usepackage{lmodern}
\usepackage{ifoddpage}
\usepackage{orcidlink}
\usepackage{amsopn}

\usepackage{comment}
\usepackage{url}

\hypersetup{
    colorlinks,
    linkcolor={blue!100!black},
    citecolor={blue!100!black},
    urlcolor={blue!80!black}
}

\newcommand{\vecn}{\textbf{n}}
\newcommand{\real}[1]{\ensuremath{\text{Re}\left[ #1 \right]}}

\newcommand{\trace}[2]{#2}

\newcommand{\compemb}{\mathrel{\mathpalette\comp@emb\relax}}

\DeclarePairedDelimiter\dualpairing{\langle}{\rangle}

\theoremstyle{plain}
\newtheorem{theorem}{Theorem}[section]

\newtheorem{proposition}[thrm]{Proposition}

\newtheorem{remark}[thrm]{Remark}

\newcommand{\revision}[1]{\textcolor{blue}{#1}}

\newcommand{\MarginA}[1]{\checkoddpage 
                        \ifoddpage
                            {\marginpar{\textcolor{blue}{\mbox{\footnotesize Rev A #1}}}}
                            \else
                            {\reversemarginpar\marginpar{\textcolor{blue}{\mbox{\footnotesize Rev A #1}}}}
                        \fi}
\newcommand{\MarginB}[1]{\checkoddpage 
                        \ifoddpage
                            {\marginpar{\textcolor{blue}{\mbox{\footnotesize Rev B #1}}}}
                            \else
                            {\reversemarginpar\marginpar{\textcolor{blue}{\mbox{\footnotesize Rev B #1}}}}
                        \fi}

\numberwithin{equation}{section}

\renewcommand\MarginA[1]{}
\renewcommand\MarginB[1]{}
\renewcommand{\revision}[1]{#1}

\begin{document}
\title{Existence, uniqueness, and numerical solutions of the nonlinear periodic Westervelt equation}
\thanks{This work received funding from the ECSEL Joint Undertaking (JU) under grant agreement No. 101007350 and was partially supported by the Austrian Science Fund (FWF) [10.55776/P36318].}
\author{Benjamin Rainer}
\address{Department of Mathematics, University of Klagenfurt, Austria, Carinthia (\email{benjamin.rainer@aau.at})\\ 
Austrian Institute of Technology GmbH, Vienna, Austria (\email{benjamin.rainer@ait.ac.at})}

\author{Barbara Kaltenbacher}
\address{Department of Mathematics, University of Klagenfurt, Austria, Carinthia (\email{barbara.kaltenbacher@aau.at}.)}
\date{Received February 7th, 2025; revised June 6th, 2025; accepted June 23rd, 2025, Erratum 4th January 2026.}
\begin{abstract}
In this paper, we study the nonlinear periodic Westervelt equation with excitations located within a bounded domain in  $\mathbb{R}^d$, where $d \in \{2,3\}$, subject to Robin boundary conditions. This problem is of particular interest for advancing imaging techniques that exploit nonlinearity of the acoustic propagation. We establish the existence and uniqueness of solutions in both the linear and the nonlinear setting, thereby allowing for spatially varying coefficients as relevant in quantitative imaging. Derivation of a multiharmonic formulation enables us to show the generation of higher harmonics (that is, responses at multiples of the fundamental  frequency) due nonlinear wave propagation. An iterative scheme for solving the resulting  system is proposed that relies on successive resolution of these higher harmonics, and its convergence under smallness conditions on the excitation is proven. Furthermore, we investigate the numerical solution of the resulting system of Helmholtz equations, employing a conforming finite element method for its discretization. Through an implementation of the proposed methodology, we illustrate how acoustic waves propagate in nonlinear media. This study aims to enhance our understanding of ultrasound propagation dynamics, which is essential for obtaining high-quality images from limited in vivo and boundary measurements.
\end{abstract}


\maketitle

\section{Introduction}
Ultrasound imaging plays a crucial role in medical diagnostics, relying on non-ionizing sound waves that pose no risk of cancer formation. In physics, ultrasound waves are characterized by a frequency above $20$ kHz and typically ranging from $2$ to $18$ MHz in medical ultrasound imaging. However, human tissue typically exhibits a nonlinear response to ultrasound already at clinically safely and commonly used pressure levels
~\cite{duck2002nonlinear}. 
This implies the necessity of an enhanced understanding of nonlinear wave propagation models, crucial in various applications such as ultrasound diagnostics for tissue discrimination. 
In this study, we investigate the existence and uniqueness of solutions to the nonlinear periodic Westervelt equation, 
governing the practically relevant scenario of periodic (e.g., sinusoidal) excitations. 

\MarginA{1.}
The nonlinear Westervelt equation reads
\begin{equation}
	\label{eq:westervelt}
	\frac{1}{c^2} p_{tt} - \Delta p - \frac{b}{c^2} \Delta p_{t} = 
	\revision{\frac{2+B/A}{2\rho_0 c^4}}
	(p^2)_{tt},
\end{equation}

\MarginA{2.}
where \revision{$B/A$ is the nonlinearity parameter\footnote{Note that $B/A$ is a single quantity -- originally derived as the quotient of two Taylor expansion coefficients -- and common notation in the nonlinear acoustics literature.}, 
	$c$ the sound speed, $b$ the diffusivity of sound and $\rho$ the mass density. All these coefficients can vary in space; their local values carry valuable diagnostic information and can be used for medical imaging, cf., e.g., \cite{Lucka2022} and the references therein.} Multiplying~\eqref{eq:westervelt} by $c^2 > 0$ and substituting 
\revision{$\eta=\frac{2+B/A}{2\rho_0 c^2}$},
equipping~\eqref{eq:westervelt} with Robin boundary as well as periodic in time (with period $T>0$) conditions 
and adding a source term $h$, we obtain the following boundary value problem for the Westervelt equation on a bounded domain $\Omega$ with $\text{C}^{1,1}$ boundary $\partial\Omega$
\begin{equation}
	\label{eq:westervelt:quadratic-nonlinearity:periodic}
	\begin{cases}
		p_{tt}(t,x) - \revision{c^2(x)} \Delta p(t,x) - \revision{b(x)}\Delta p_t(t,x) = \eta(x) (p(t,x)^2)_{tt} + h(t,x) & \text{in}\, (0,T) \times \Omega,\\
		\beta p_t(t,x) + \gamma p(t,x) + \nabla p(t,x) \cdot \vecn = 0 & \text{on}\,  (0,T) \times \partial\Omega , \\
		p(0,x) = p(T,x), \, p_t(0,x) = p_t(T,x) & x \in \Omega, \\
	\end{cases}   
\end{equation}
where $\beta,\gamma \geq  0$ are the parameters for specifying absorption or impedance conditions on $\partial\Omega$, and $\vecn$ denotes the outer normal on $\partial \Omega$. 

In our analysis we will assume $\eta\revision{,\,c,\, b} \in L^\infty(\Omega)$, noting that (possibly nonsmoothly) varying $B/A$\revision{, $c$, $b$} is highly relevant for imaging techniques that rely on the tissue specific values of \revision{these parameters}. 
Indeed, there exists a vast corpus of literature, e.g.,~\cite{nonlinparam1, duck2002nonlinear,  nonlinparam3, nonlinparam2, ZHANG20011359} that puts the value of the nonlinearity parameter in relation to different tissue types. Hence, the nonlinearity parameter $\eta(x)$ can be used as an imaging quantity and its reconstruction in inhomogeneous media from (incomplete) measurements is of high interest.

The SI units of the function $h(t,x)$ in \eqref{eq:westervelt:quadratic-nonlinearity:periodic} are 
Pa/s$^2$ and 
thus it corresponds to the
second time derivative of a periodic space and time dependent pressure source $g$. That is, with $g(t,x) = \text{Re}\{\hat{h}(x) e^{\iota \omega t}\}$,  $\omega = \frac{2 \pi}{T}$, we have $h(t,x) = g_{tt}(t,x)$. 
Due to the nonlinearity, solutions to~\eqref{eq:westervelt:quadratic-nonlinearity:periodic} will exhibit contributions at the fundamental frequency $\omega$ and multiples thereof -- so-called higher harmonics. 
This will in fact be a key element in our reformulation of \eqref{eq:westervelt:quadratic-nonlinearity:periodic}.

\MarginA{3.}
\revision{
	\begin{remark}
		In view of the fact that in actual ultrasound imaging, the excitation would be provided by a source located on the boundary, modeling the excitation by an interior rather than a boundary source is a model simplification that we make in order to render the analysis somewhat less technical and proving well-posedness with boundary instead of interior inhomogeneity is subject of future research.\\
		Still, the generation of interior sources by the nonlinear interaction through the quadratic term is one of the key intuitions behind imaging with nonlinear waves and can be explained by a Born approximation type argument as follows.
		\def\src{s}
		Considering \eqref{eq:westervelt} with nonhomogeneous boundary source $\src$ 
		\begin{equation*}
			\begin{cases}
				\frac{1}{c^2} p_{tt} - \Delta p - \frac{b}{c^2} \Delta p_{t} = 
				\frac{2+B/A}{2\rho_0 c^4}
				(p^2)_{tt} & \text{in}\, (0,T) \times \Omega,\\
				\beta p_t + \gamma p + \nabla p \cdot \vecn = \src & \text{on}\,  (0,T) \times \partial\Omega 
			\end{cases}
		\end{equation*}
		and decomposing the pressure into a part $p_\src$ satisfying the linear wave equation with these boundary conditions 
		\begin{equation*}
			\begin{cases}
				\frac{1}{c^2} p_{\src\,tt} - \Delta p_\src - \frac{b}{c^2} \Delta p_{\src\,t} = 0
				& \text{in}\, (0,T) \times \Omega,\\
				\beta p_{\src\,t} + \gamma p_\src + \nabla p_\src \cdot \vecn = \src & \text{on}\,  (0,T) \times \partial\Omega 
			\end{cases}
		\end{equation*}
		and the rest $\tilde{p}:=p-p_\src$ it becomes apparent that $\tilde{p}$ satisfies a problem with interior source $\tilde{g}:=\frac{2+B/A}{2\rho_0 c^4} (p_\src)^2$ and a modified sound speed defined by $\frac{1}{\tilde{c}^2}=\frac{1}{c^2}-\frac{2+B/A}{\rho_0 c^4}p_\src$ 
		\begin{equation*}
			\begin{cases}
				\left(\frac{1}{\tilde{c}^2} \tilde{p}\right)_{tt} - \Delta \tilde{p} - \frac{b}{c^2} \Delta \tilde{p}_{t} = 
				\frac{2+B/A}{2\rho_0 c^4}
				(\tilde{p}^2)_{tt} \, + \tilde{g}_{tt}& \text{in}\, (0,T) \times \Omega,\\
				\beta \tilde{p}_t + \gamma \tilde{p} + \nabla \tilde{p} \cdot \vecn = 0 & \text{on}\,  (0,T) \times \partial\Omega 
			\end{cases}
		\end{equation*}
	\end{remark}
}

Taking a closer look at~\eqref{eq:westervelt:quadratic-nonlinearity:periodic} and applying the identity $(p(t,x)^2)_{tt} = 2(p_t(t,x)^2 + p(t,x) p_{tt}(x,t))$ we obtain
\begin{equation}
	\label{degeneracy}
	(1- 2 \eta(x)p(t,x))p_{tt}(t,x) - \revision{c^2(x)} \Delta p(t,x) - \revision{b(x)}\Delta p_t(t,x) = 2\eta(x) p_t(t,x)^2 + h(t,x).
\end{equation}
This shows that the Westervelt equation degenerates if $1-2\eta(x)p(t,x) \leq 0$ for some $x \in \Omega$, which makes its analysis challenging. 
Several studies, such as ~\cite{KL09Westervelt,MizohataUkai,NIKOLIC20151131} have established local existence and uniqueness results for initial value problems for the Westervelt equation with different boundary conditions, as well as a priori estimates for finite element methods (FEM)~\cite{Nikolic2019}. Numerical methods, including time-stepping methods with fractional damping~\cite{baker2024numerical} and finite difference approaches~\cite{solovchuk2013simulation} offer simulation based insights into the nonlinear propagation of ultrasound. Additionally, analytical solutions for specific boundary problems, like those presented in~\cite{Jan2022} contribute to our understanding of the spatial distribution of harmonic components.

Since excitations are often sinusoidal in ultrasound applications, the development of tools for their analysis and numerical simulation in frequency domain is highly desirable to enhance efficiency and enable simulation times and precision guarantees that are compatible with the demands of imaging tasks.
The difficulty here lies in the fact that the quadratic nonlinearity gives rise to auto convolution terms in the Fourier transformed (with respect to time) equations and thus disables control of the potential degeneracy the way it is done in the time domain analysis, namely by bounding a pressure dependent coefficient (cf. $1-2\eta p$ in~\eqref{degeneracy}).

Following up on the preliminary analysis in~\cite{Kaltenbacher2021}, we therefore 
\begin{itemize}
	\item provide a well-posedness analysis that allows to deal with the potential degeneracy in frequency domain, cf. Theorem~\ref{thm:westervelt:nonlinear:solution:existence:uniqueness};
	\item develop an iterative multi-level (with respect to frequency levels) method for solving~\eqref{eq:westervelt:quadratic-nonlinearity:periodic} in frequency domain and prove its convergence, cf. Theorem~\ref{thm:iteration:scheme:convergence};
	\item devise a finite element scheme for its implementation and provide experiments demonstrating its efficiency.
\end{itemize}
As compared to \cite{Kaltenbacher2021}, we carry out our  analysis relying on a perturbation of the strongly damped wave equation with unit inertial coefficient (rather than a space-time varying one), which leads to a fixed point scheme that is much more efficient to implement. Moreover, we consider a setting incorporating spatially (possibly non-smoothly) varying sound speed, attenuation and nonlinearity coefficients, as highly relevant in imaging.

This paper is structured as follows: we introduce notation and relevant inequalities in~\cref{sec:notation}; lay the theoretical well-posedness groundwork in~\cref{sec:main_existence_uniqueness} by establishing existence and uniqueness of solutions; present a multiharmonic expansion of the nonlinear periodic Westervelt equation including an iteration scheme and its convergence analysis for numerical computations in~\cref{sec:iteration:scheme}; provide experimental results in~\cref{sec:experiments}; and draw some conclusions in~\cref{sec:conclusions}.

\section{Notation and auxiliary results}
\label{sec:notation}
We briefly set the notation and recall important inequalities used throughout this work. The (fractional) Sobolev spaces over an open and bounded $\text{C}^1$ domain $\Omega$ is denoted by $W^{s,p}(\Omega)$ with $1\leq p \leq \infty$, $s \in \mathbb{R}_{\geq 0}$ and by real interpolation we have $W^{s,p}(\Omega) = [W^{k,p}(\Omega),W^{k+1,p}(\Omega)]_{\theta,p}$, where $k = \lfloor s \rfloor$, and $\theta = s - \lfloor s \rfloor$. The Hilbert space $W^{s,2}(\Omega)$ is denoted by $H^{s}(\Omega)$. We denote the norm on these spaces by $||\cdot||_{W^{s,p}(\Omega)}$. On the involved Bochner spaces $L^p(0,T;L^q(\Omega))$ we sometimes abbreviate the norm by $||\cdot||_{L^p(L^q(\Omega))}$. Dependencies on time and space variables will be skipped for the sake of readability whenever they are clear from the context. 
For an open, bounded and connected $\text{C}^1$ domain $\Omega \subseteq \mathbb{R}^d$, $d \in \{2,3\}$ we make use of the following Sobolev embeddings and the resulting inequalities
\begin{align}
	\label{sobolev:space:1}
	||u||_{L^\infty(\Omega)} \leq C_{H^{13/8}(\Omega) \rightarrow L^\infty(\Omega)} ||u||_{H^{13/8}(\Omega)},
\end{align}
\begin{align}
	\label{sobolev:space:2}
	||u||_{L^4(\Omega)} \leq C_{H^{9/8}(\Omega) \rightarrow L^4(\Omega)} ||u||_{H^{9/8}(\Omega)},
\end{align}
%

\MarginB{iii.}
In the space-time domain we use the embedding inequalities 
\begin{align}
	\label{sobolev:space:time:1}
	||u||_{L^\infty(0,T;X)} \leq C_{H^{3/4}(0,T) \rightarrow L^\infty(0,T)} ||u||_{H^{3/4}(0,T; X)},
\end{align}
\begin{align}
	\label{sobolev:space:time:2}
	||u||_{L^4(0,T;X)} \leq C_{H^{1/4}(0,T) \rightarrow L^4(0,T)} ||u||_{H^{1/4}(0,T;X)}.
\end{align}
Moreover, we apply the trace theorem for $\Omega \subset \mathbb{R}^d$, open, bounded, connected, $d \in \{2,3\}$ with \revision{at least} Lipschitz boundary $\partial\Omega$, which yields 
\begin{align}
	\label{eq:trace:theorem:H1}
	||\trace{\Omega}{u}||_{H^{1/2}(\partial\Omega)} \leq C_{\text{trace}} ||u||_{H^1(\Omega)},
\end{align}
and the Poincaré-Friedrichs inequality which reads
\begin{equation}
	\label{eq:poincare:friedrich:inequality}
	||u||_{H^1(\Omega)}^2 \leq C_{\text{PF}}^2 \left(||\nabla u||_{L^2(\Omega)}^2 + ||\trace{\Omega}{u}||_{L^2(\partial \Omega)}^2\right).
\end{equation}
\section{Existence and uniqueness of periodic solutions}
\label{sec:main_existence_uniqueness}
In this section we investigate the existence and uniqueness of a solution to~\eqref{eq:westervelt:quadratic-nonlinearity:periodic}. 
We do so by first establishing well-posedness of a linear problem with variable coefficients and then extending this result towards the nonlinear setting for a sufficiently small excitation by means of a fixed point argument. Our results extend those from~\cite{Kaltenbacher2021} to spatially variable coefficients $c$, $b$, and $\eta$, as relevant for imaging.
Minimizing the smoothness assumptions on these coefficients to allow for jumps requires a dedicated weak and Galerkin formulation, cf. \eqref{eq:westervelt:linear:weak:form}, \eqref{eq:westervelt:linear:galerkin:form:full:ODE:system:a} in the proof of Theorem~\ref{thm:westervelt:linear:solution:existence:uniqueness}.
Note that a strictly positive impedance parameter $\gamma>0$ is needed to avoid constant solutions of the homogeneous problem, that is, to obtain uniqueness.
\begin{theorem}
	\label{thm:westervelt:linear:solution:existence:uniqueness}
	[Linear well-posedness]\\
	Let $T>0$, $\Omega \subseteq \mathbb{R}^d$, $d \in \{2,3\}$, open, bounded, connected, with $C^{1,1}$ boundary, $\beta,\,\gamma>0$, $c, \, b,\,\frac{1}{c},\frac{1}{b}\in L^\infty(\Omega)$, $c,\, b>0$ 
	$\alpha, \, \frac{1}{\alpha}\, \in L^\infty (0,T;L^\infty(\Omega)), \, \alpha_t\, \in L^\infty (0,T;L^\infty(\Omega))$, $\alpha>0$, $\alpha(0) = \alpha(T)$, $\mu \in C([0,T];L^{2q/(q-1)}(\Omega))$, $\mu(0) = \mu(T)$
	for some $q\in[1,\infty)$ and 
	$f \in L^2(0,T;L^2(\Omega))$
	
	Then there exists $r_0>0$ such that for coefficients satisfying the smallness condition
	\begin{equation}\label{smallness}
		\|\alpha_t\|_{L^\infty(0,T;L^\infty(\Omega))}
		+\|\tfrac{\gamma}{\beta}-\tfrac{\mu}{\alpha}\|_{L^\infty(0,T;L^{2q/(q-1)}(\Omega))}
		<r_0
	\end{equation} 
	there exists a unique (weak) solution $u$ of
	\begin{equation}
		\label{eq:westervelt:linear:periodic}
		\begin{cases}
			\alpha u_{tt} - c^2 \Delta u - b \Delta u_t + \mu u_t= f & \text{in } \Omega\times(0,T),\\
			\beta u_t + \gamma u + \nabla u \cdot \vecn = 0 & \text{on } \partial\Omega\times(0,T), \\
			u(0) = u(T), \, u_t(0) = u_t(T) & \text{in } \Omega, \\
		\end{cases}   
	\end{equation}
	with 
	\begin{align}
		\nonumber
		u \in X =\{v \in H^2(0,T;L^2(\Omega)) \cap H^1(0,T;H^{3/2}(\Omega)) \cap L^2(0,T;H^2(\Omega)):  \\ \nonumber
		|| \trace{\Omega}{\nabla v\cdot\mathbf{n}} ||_{H^1(0,T;L^2(\partial\Omega))} < \infty, v(0) = v(T), v_t(0) = v_t(T) \text{ a.e.}\}
	\end{align}
	\MarginA{4.}
	and the solution $u$ satisfies
	\begin{align}
		\nonumber
		||u||_X^2 \leq C^2 ||f||_{ L^2(0,T;L^2(\Omega))}^2,
	\end{align}
	where
	\begin{align}
		\nonumber
		||u||_X^2  = ||u||_{H^2(0,T;L^2(\Omega))}^2 + ||u||_{H^1(0,T;H^{3/2}(\Omega))}^2 + ||u||_{L^2(0,T;H^2(\Omega))}^2 + || \trace{\Omega}{\nabla v\cdot\mathbf{n}} ||_{H^1(0,T;L^2(\partial\Omega))}^2,
	\end{align}
	and $C=C(\alpha, b,\gamma,c,\beta, T, \Omega) > 0$.
\end{theorem}
\MarginB{iv.}
\begin{proof}
	See the~\cref{appendix:proof:thm:3:1}.
	\revision{As can be seen from this proof, dependency of $C(\alpha, b,\gamma,c,\beta, T, \Omega)$ on the coefficients $\alpha,\,b,\,c$ is actually only via the $L^\infty$ norms of $\alpha, \, \frac{1}{\alpha}\, , \alpha_t$, $c, \, b,\,\frac{1}{c},\frac{1}{b}\in L^\infty(\Omega)$.}
\end{proof}
\MarginA{5.}
%
The solution space established in Theorem~\ref{thm:westervelt:linear:solution:existence:uniqueness}, given the assumed regularity of the coefficients and the Sobolev embedding theorem, implies that a solution to problem~\eqref{eq:westervelt:linear:periodic} is continuous in time, i.e., an element of $C(0,T;H^{3/2}(\Omega)) \cap C^1(0,T;L^2(\Omega))$. While periodicity of the right-hand side $f$ is not required, it is essential to use periodic coefficients $\alpha$ and $\mu$.

\begin{theorem}
	\label{thm:westervelt:nonlinear:solution:existence:uniqueness}[Nonlinear well-posedness]\\
	Let $T>0$, $\Omega \subseteq \mathbb{R}^d$, $d \in \{2,3\}$, open, bounded, connected, with $C^{1,1}$ boundary, $\beta,\,\gamma>0$, $c, \, b,\,\frac{1}{c},\frac{1}{b}\in L^\infty(\Omega)$, $c,\, b>0$ and $\eta \in L^\infty(\Omega)$. 
	Then there exists $\delta > 0$ such that for all $h\in L^2(0,T;L^2(\Omega))$ with $||h||_{L^2(0,T;L^2(\Omega))} \leq \delta$ there exists a unique (weak) solution $u \in X$ of
	\begin{equation}
		\label{eq:westervelt:nonlinear:periodic}
		\begin{cases}
			u_{tt} - c^2 \Delta u - b \Delta u_t = \eta (u^2)_{tt} + h & \text{in }  (0,T) \times \Omega,\\
			\beta u_t + \gamma u + \nabla u \cdot \vecn = 0 & \text{on } (0,T) \times \partial\Omega, \\
			u(0) = u(T), \, u_t(0) = u_t(T) & \text{in } \Omega, \\
		\end{cases}   
	\end{equation}
	and the solution $u$ satisfies
	\begin{align}
		\nonumber
		||u||_X \leq C ||h||_{ L^2(0,T;L^2(\Omega))},
	\end{align}
	for some constant $C=\revision{C(\delta, b,\gamma,c,\beta, T, \Omega)} > 0$ independent of $h$ \revision{and u}. 
\end{theorem}
\begin{proof}
	See the~\cref{appendix:proof:thm:3:2}.
\end{proof}
\MarginA{5.}
%
%
Theorem~\ref{thm:westervelt:nonlinear:solution:existence:uniqueness} establishes the theoretical foundation for studying the periodic Westervelt problem. However, its proof -- though being constructive in principle, by applying a fixed point argument -- does not provide an efficient numerical method for obtaining the solution. The continuity and periodicity of the solution naturally leads to a multiharmonic ansatz, which we leverage in the next section to develop a multilevel scheme for the numerical computation of solutions.
\section{Iteration scheme exploiting a multiharmonic expansion}
\label{sec:iteration:scheme}
Due to the periodicity condition in~\eqref{eq:westervelt:quadratic-nonlinearity:periodic}, the solution can be represented in terms of the $L^2(0,T)$ orthogonal basis $(e^{\iota m \omega t})_{m\in\mathbb{Z}}$ as $p(t,x) = \text{Re}[\sum_{m = 0}^{\infty} \hat{p}_{\revision{m}}(x) e^{\iota m \omega t}]$, cf.~\cite{Kaltenbacher2021}.
Motivated by this, we use a multiharmonic ansatz 
\begin{equation}\label{pNansatz}
	p^N(t,x) = \text{Re}[\sum_{m = 0}^{N} \hat{p}_{\revision{m}}(x) e^{\iota m \omega t}], 
	\text{ where } 
	\hat{p}_{\revision{m}} \in L^2(\Omega; \mathbb{C})
\end{equation} 
\MarginA{6.}
and  project~\eqref{eq:westervelt:quadratic-nonlinearity:periodic} onto the 
subspace 
\begin{equation}
	\nonumber
	X_N:= \left\{\text{Re}  \left[  \sum_{m=0}^N \alpha_m(x) e^{\iota m \omega t} \right]: \alpha_m \in H^2(\Omega; \mathbb{C}) \right\}
\end{equation} 
to obtain an approximation of the actual solution. 
We will further investigate under which assumptions this projection converges to the unique solution of~\eqref{eq:westervelt:quadratic-nonlinearity:periodic}. 
The identity $\text{Re}[z] = \frac{1}{2}(z + \overline{z})$ applied to our ansatz reads
\begin{equation}
	\nonumber
	\text{Re}\left[\sum_{m=0}^N \hat{p}_{\revision{m}}(x) e^{\iota m \omega t} \right] = \frac{1}{2} \left(\sum_{m=0}^N \hat{p}_{\revision{m}}(x) e^{\iota m \omega t} + \sum_{m=0}^N \overline{\hat{p}_{\revision{m}}(x)} e^{- \iota m \omega t} \right),
\end{equation}
\MarginA{7.}
and, together with $L^2$ orthogonality of trigonometric functions (conveniently expressed via complex exponential functions) is one of the tools for deriving a multiharmonic system by plugging \eqref{pNansatz} into \eqref{eq:westervelt:quadratic-nonlinearity:periodic}.

\begin{proposition}
	\label{prop:helmholtz:fully:coupled}
	Projecting~\eqref{eq:westervelt:quadratic-nonlinearity:periodic} onto $X_N$ with periodic excitation $h(t,x)=g_{tt}$, $g(t,x) = \text{Re}[\sum_{k=1}^N \hat{h}_k(x) e^{\iota k \omega t}]$, $\omega = \frac{2 \pi}{T}$ yields the following coupled system of Helmholtz equations for $x \in \Omega$:
	\begin{equation}
		\label{eq:helmholtz:systems}
		\begin{cases}
			-c^2\Delta\hat{u}_0(x) = 0 & m=0, \\ 
			-\kappa^2 m^2 \hat{u}_m(x) - \Delta \hat{u}_m(x) = \\ 
			\qquad \qquad - \frac{m^2 \kappa^2 \eta(x)}{2} \left(\sum_{j=1}^m \hat{u}_j(x) \hat{u}_{m-j}(x) + 2 \sum_{j=m:2}^{2N-m}\overline{\hat{u}_{\frac{j - m}{2}}(x)}\hat{u}_{\frac{j + m}{2}}(x) \right) - \kappa^2 m^2 \hat{h}_m(x)& 
			m\in\mathbb{N}
		\end{cases}   
	\end{equation}
	where $\kappa^2=\frac{\omega^2}{c^2 + \iota m \omega b}$. 
	\revision{The notation $:2$ in the sum indicates that the index takes steps of size two, that is, $j\in\{m,m+2,m+4,\ldots, m+2(N-m)\}$.} 
	
	This system is further equipped with the projected boundary conditions $\left(\iota \beta m \omega + \gamma \right) \hat{u}_m + \nabla \hat{u}_m \cdot \vecn = 0$ on $\partial \Omega$.
\end{proposition}
\MarginB{v.}
\begin{proof}
	See the~\cref{appendix:proof:prop:4:1}.
\end{proof}

\MarginA{8.}
The system in~\eqref{eq:helmholtz:systems} is fully coupled, and the right-hand side scales with $m^2$ exhibiting an autoconvolution that blends together higher and lower harmonics of the solution. This behavior arises from the nonlinear nature of the Westervelt equation, as demonstrated in the proof of Proposition~\ref{prop:helmholtz:fully:coupled}. 
\revision{Moreover, the boundary condition on each harmonic of the solution further confirms the validity and necessity of assuming $\gamma > 0$ to ensure a unique solution. 
	(i.e., the case $m=0$ is equipped with the boundary condition $ \gamma \hat{u}_m + \nabla \hat{u}_m \cdot \vecn = 0$ on $\partial \Omega$).}
The found system~\eqref{eq:helmholtz:systems} is too large to be directly discretized (e.g., by the finite element methods), since it is fully coupled. From density of the union of the spaces $X_N$ in $X$ we conclude that for a solution $u$ to~\eqref{eq:westervelt} (in fact, for any $u\in X$), 
\begin{equation}
	|| u - \text{Proj}_{X_N}u||_X \rightarrow 0, \text{ as } N \rightarrow \infty.
\end{equation}
To numerically approximate the solution, we adopt the fixed-point approach from Theorem~\ref{thm:westervelt:nonlinear:solution:existence:uniqueness} to construct an iterative approximation scheme, but do so in a multilevel fashion, sucessively increasing the frequency level, by defining $\Tilde{u}_N \in X_N$ as the solution to
\begin{align}
	\label{eq:projection:iteration:scheme}
	\text{Proj}_{X_N}\left((\Tilde{u}_N)_{tt} - c^2 \Delta (\Tilde{u}_N) - b \Delta (\Tilde{u}_N)_t - \eta \left(\Tilde{u}_{N-1}^2\right)_{tt} - h \right) = 0.
\end{align}
Note that the nonlinear term is only considered on the previous frequency level and therefore only a system of linear Helmholtz equations need to be solved in each step, which due to independence of the coefficients on the states $\Tilde{u}_{N-1}$ can be done in a highly efficient parallelized manner. Convergence of this scheme is stated the following theorem.
\begin{theorem}
	\label{thm:iteration:scheme:convergence}
	[Iterative approximation scheme]\\    
	Under the assumptions of Theorem~\ref{thm:westervelt:nonlinear:solution:existence:uniqueness} there exists $\delta>0$ such that for all $h \in L^2(0,T,L^2(\Omega))$ with $||h||_{L^2(0,T;L^2(\Omega))} < \delta$, 
	the sequence of solutions 
	$\Tilde{u}_N \in X_N$ of the iteration scheme \eqref{eq:projection:iteration:scheme}
	converges to the unique solution $u$ of~\eqref{eq:westervelt:nonlinear:periodic}.
\end{theorem}
\begin{proof}
	Consider the operator $L_{\alpha,\mu}: X \rightarrow X^{\prime}$ and $\Tilde{u}_N \in X_N$ as a Galerkin solution of
	\begin{align}
		\label{eq:projection:iteration:scheme:proof}
		&\dualpairing{L_{\alpha,\mu} \Tilde{u}_N - f,\phi}_{X^\prime,X} := \int_{0}^T \left(\alpha (\Tilde{u}_N)_{tt} + \mu (\Tilde{u}_N)_t, \phi \right)_{L^2(\Omega)} + \left(c^2 \nabla \Tilde{u}_N + b \nabla (\Tilde{u}_N)_t,    \nabla \phi \right)_{L^2(\Omega)}\\ \nonumber
		& + \left((c^2 \beta + b\gamma) (\Tilde{u}_N)_t + c^2 \gamma \Tilde{u}_N + b \beta (\Tilde{u}_N)_{tt}, \phi \right)_{L^2(\partial\Omega)} - (f,\phi)_{L^2(\Omega)} dt = 0,
	\end{align}
	for all $\phi \in X_N$ where we set $\alpha = 1$, $\mu = 0$, and $f = \eta \left(\Tilde{u}_{N-1}^2\right)_{tt} + h$. 
	One readily checks that the solution to~\eqref{eq:projection:iteration:scheme:proof} exists and is unique using the respective test functions $\Tilde{u}_N$, $(\Tilde{u}_N)_t$, $(\Tilde{u}_N)_{tt}$, $-\Delta\Tilde{u}_N$, $-\Delta(\Tilde{u}_N)_t$ which are indeed contained in $X_N$ for $\alpha$ and $f$ fulfilling the assumptions of Theorem~\ref{thm:westervelt:linear:solution:existence:uniqueness}. 
	The coefficient functions $\hat{u}_{k}^{N}$ in the representation $\Tilde{u}_{N} = \sum_{k=1}^N \hat{u}_{k}^{N} e^{\iota k \omega t}$ satisfy the coupled Helmholtz system
	\begin{equation}
		\label{eq:helmholtz:systems:iteration:scheme}
		\begin{cases}
			-c^2\Delta\hat{u}_0^{N}(x) = 0 \\
			-\kappa^2 m^2 \hat{u}_m^{N}(x) - \Delta \hat{u}_m^{N}(x) = \\ 
			\quad - \frac{\kappa^2 m^2 \eta(x)}{2} \left(\sum_{j=1}^m \hat{u}_j^{N-1}(x) \hat{u}_{m-j}^{N-1}(x) + 2 \sum_{j=m:2}^{2N-m-2}\overline{\hat{u}_{\frac{j - m}{2}}^{N-1}(x)}\hat{u}_{\frac{j + m}{2}}^{N-1}(x) \right) - \kappa^2 m^2 \hat{h}_m(x)  
			\qquad m \in\mathbb{N},
		\end{cases}   
	\end{equation}
	with the boundary conditions $\iota \beta m \omega \hat{u}_m^{N} + \gamma \hat{u}_m^{N} + \nabla \hat{u}_m^{N} \cdot \vecn =0$ on $\partial \Omega$.
	We obtain $||\Tilde{u}_{N}||_X < r$ provided $||\Tilde{u}_{N-1}||_X < r$, by~\eqref{eq:westervelt:nonlinear:condition:3:modified} for $||h||_{L^2(0,T;L^2(\Omega))} < \delta$ and $r > 0$ chosen small enough according to 
	\revision{ the arguments in the proof of }
	Theorem~\ref{thm:westervelt:nonlinear:solution:existence:uniqueness}. 
	\MarginB{vi.}
	A direct investigation of the error $w:= \Tilde{u}_{N} - u$ is not possible because the difference does not need to be contained in $X_N$. Therefore, we circumvent this issue by investigating $\hat{w}:= \Tilde{u}_{N} - \text{Proj}_{X_N}u \in X_N$. For $\phi \in X_N$ we have
	\begin{align}
		\nonumber
		\dualpairing{L_{\alpha,\mu} \hat{w},\phi}_{X^\prime,X} & =  \dualpairing{L_{\alpha,\mu} (\Tilde{u}_{N} - u + u - \text{Proj}_{X_N}u ),\phi}_{X^\prime,X} \\ \nonumber
		& = \dualpairing{L_{\alpha,\mu} (\Tilde{u}_{N} - u) +  L_{\alpha,\mu} ( u - \text{Proj}_{X_N}u),\phi}_{X^\prime,X} \\ \nonumber
		& =  \left( \eta (\Tilde{u}_{N-1}^2)_{tt}  - \eta(u^2)_{tt}, \phi \right)_{L^2(0,T;L^2(\Omega))} + \dualpairing{L_{\alpha,\mu} (u - \text{Proj}_{X_N}u) ,\phi}_{X^\prime,X} \\ \nonumber
		& =  \left(2\eta [ ( (\Tilde{u}_{N-1})^2_{t} - u_t ^2) + \Tilde{u}_{N-1} ( (\Tilde{u}_{N-1})_{tt} - u_{tt}) \right.  \\ \nonumber
		& \left. + u_{tt}(\Tilde{u}_{N-1} - u) ], \phi \right)_{L^2(0,T;L^2(\Omega))} + \dualpairing{L_{\alpha,\mu} (u - \text{Proj}_{X_N}u) ,\phi}_{X^\prime,X},
	\end{align}
	where we used the fact that $(L_{\alpha,\mu} u, \phi)_{L^2(0,T;L^2(\Omega))} = (\eta (u^2)_{tt} + h, \phi)_{L^2(0,T;L^2(\Omega))}$, $\forall \phi \in X$ and $X_N \subset X$.
	%
	Theorem~\ref{thm:westervelt:linear:solution:existence:uniqueness} yields a constant $C>0$ such that
	\begin{align}
		\nonumber
		||\Tilde{u}_{N} - \text{Proj}_{X_N}u||_X \leq C || & 2\eta [ (  (\Tilde{u}_{N-1})^2_{t} - u_t ^2) + \Tilde{u}_{N-1} ( (\Tilde{u}_{N-1})_{tt} - u_{tt}) \\ \nonumber
		& + u_{tt}(\Tilde{u}_{N-1} - u) ] + L_{\alpha,\mu}(u - \text{Proj}_{X_N}u)||_{L^2(0,T;L^2(\Omega))}.
	\end{align}
	We estimate similarly as in Theorem~\ref{thm:westervelt:nonlinear:solution:existence:uniqueness} and obtain
	\begin{align}
		\nonumber
		&C ||2\eta [ (\Tilde{u}_{N-1})^2_{t} - u_t ^2 + \Tilde{u}_{N-1} ((\Tilde{u}_{N-1})_{tt} - u_{tt}) + u_{tt}(\Tilde{u}_{N-1} - u)]||_{L^2(0,T;L^2(\Omega))} \\ \nonumber
		& \leq 2 C ||\eta||_{L^\infty(\Omega)} (||(\Tilde{u}_{N-1})_{t} - u_{t}||_{L^4(0,T;L^4(\Omega))} ||(\Tilde{u}_{N-1})_{t} + u_{t}||_{L^4(0,T;L^4(\Omega))}  \\ \nonumber
		& \qquad  + ||\Tilde{u}_{N-1}||_{L^\infty(0,T;L^\infty(\Omega))} ||(\Tilde{u}_{N-1} - u)_{tt}||_{L^2(0,T;L^2(\Omega))}   + ||u_{tt}||_{L^2(0,T;L^2(\Omega))} ||\Tilde{u}_{N-1} - u||_{L^\infty(0,T;L^\infty(\Omega))} ) \\ \nonumber
		& \leq  rC_0  ||\Tilde{u}_{N-1} - u||_X.
	\end{align}
	Since $\bigcup_{N \in \mathbb{N}} X_{N}$ is dense in $X$, it holds that
	\begin{align}
		\nonumber
		||L_{\alpha,\mu}(u - \text{Proj}_{X_N} u)||_{X^\prime} \leq ||L_{\alpha,\mu}||_{X,X^\prime} || u -  \text{Proj}_{X_N} u||_X \rightarrow 0, \, N \rightarrow \infty.
	\end{align}
	According to Theorem~\ref{thm:westervelt:nonlinear:solution:existence:uniqueness} we have $ q(r) := rC_0 < 1$ and by the triangle inequality we obtain
	\begin{align}
		\nonumber
		|| \Tilde{u}_{N} - u||_X \leq q(r) || \Tilde{u}_{N-1} - u ||_X + (C+1) || u -  \text{Proj}_{X_N} u||_X.
	\end{align}
	By induction, $\Tilde{u}_{0} = 0$ and setting $\Tilde{C} = C + 1$ we finally arrive at
	\begin{align}
		\nonumber
		|| \Tilde{u}_{N} - u ||_X \leq q(r)^N ||u||_X + \Tilde{C} \sum_{i=1}^N q(r)^{N-i} || u -  \text{Proj}_{X_{i}} u||_X.
	\end{align}
	Now, let $\varepsilon > 0$ be arbitrary, by the density of $\bigcup_{N \in \mathbb{N}} X_{N}$ in $X$ there exists an $N(\varepsilon) \in \mathbb{N}$ such that for all $n > N(\varepsilon)$ we have $||u - \text{Proj}_{X_{n}}u||_X < \varepsilon$. We set $N > N(\varepsilon)$, large enough such that $q(r)^{N-N(\varepsilon)} < \varepsilon$ holds. This yields 
	\begin{align}
		\nonumber
		|| \Tilde{u}_{N} - u ||_X & \leq q(r)^N ||u||_X + \Tilde{C} \sum_{i=1}^N q(r)^{N-i} || u -  \text{Proj}_{X_{i}}u||_X \\ \nonumber
		& = q(r)^N ||u||_X  + \Tilde{C} q(r)^{N - N(\varepsilon)} \sum_{i=1}^{N(\varepsilon)} q^{N(\varepsilon)-i} || u -  \text{Proj}_{X_{i}}u||_X  \\ \nonumber
		& \qquad \qquad \qquad + \Tilde{C} \sum_{i=N(\varepsilon)+1}^N q(r)^{N-i} || u -  \text{Proj}_{X_{i}}u||_X \\ \nonumber
		& < \varepsilon ||u||_X + \frac{\varepsilon \Tilde{C} c_{\text{Proj}}}{1-q}  + \frac{ \varepsilon \Tilde{C}}{1-q}  =  \varepsilon \left(||u||_X + \frac{\Tilde{C}\left( 1 + c_{\text{Proj}} \right) }{1-q}  \right),
	\end{align}
	where $c_{\text{Proj}}:= \sup_{i \in \mathbb{N}} || u -  \text{Proj}_{X_{i}}u||_X$ is finite because $\bigcup_{N \in \mathbb{N}} X_{N}$ is dense in $X$. This proves the convergence of $\Tilde{u}_{N}$ to $u$ in X.
\end{proof}
%
\section{Numerical experiments}
\label{sec:experiments}
\MarginB{vii.}
We revisit the result from Theorem~\ref{thm:iteration:scheme:convergence} where we obtained~\eqref{eq:helmholtz:systems:iteration:scheme} which simplifies to the following system of Helmholtz equations for $x \in \Omega$:
\begin{equation}
	\label{eq:helmholtz:systems:iteration:scheme:second:mentioning}
	\begin{cases}
		-\kappa^2 \hat{u}_1^{N}(x) - \Delta \hat{u}_1^{N}(x) = -\kappa^2 \left(\eta(x) \sum_{j=1:2}^{2N-3}\revision{\overline{\hat{u}_{\frac{j-1}{2}}^{N-1}(x)}}\hat{u}_{\frac{j + 1}{2}}^{N-1}(x) +\hat{h}_1(x)\right) & m = 1, \\
		-\kappa^2 m^2 \hat{u}_m^{N}(x) - \Delta \hat{u}_m^{N}(x) = \\ 
		-\kappa^2 m^2 \hat{h}_m(x) - \frac{\kappa^2  m^2 \eta(x)}{2} \left(\sum_{j=1}^m \hat{u}_j^{N-1}(x) \hat{u}_{m-j}^{N-1}(x) + 2 \sum_{j=m:2}^{2(N-1)-m}\revision{\overline{\hat{u}_{\frac{j - m}{2}}^{N-1}(x)}}\hat{u}_{\frac{j + m}{2}}^{N-1}(x) \right)   & m > 1,
	\end{cases}   
\end{equation}

with the following boundary conditions for $x \in \partial\Omega:$
\begin{align}
	\iota \beta m \omega \hat{u}_m^{N}(x) + \gamma \hat{u}_m^{N}(x) + \nabla \hat{u}_m^{N} (x)\cdot \vecn =0,
\end{align}
where $c, \omega, b >0$, $\kappa^2 = \frac{\omega^2}{c^2 + \iota m \omega b} \in \mathbb{C}$, $\Omega \subset \mathbb{R}^2$ with a $C^{1,1}$ boundary $\partial\Omega$ and in case of non constant $c(x)$, $b(x)$ we assume that $c, \frac{1}{c}, b, \frac{1}{b} \in L^\infty(\Omega)$ and obtain a space dependent wave number $\kappa(x)^2 = \frac{\omega^2}{c^2(x) + \iota \omega b(x)}$. The solution for $m=0$ is zero. In Theorem~\ref{thm:iteration:scheme:convergence} we have shown that this iteration scheme converges to the unique (weak) solution of~\eqref{eq:westervelt} as $N \rightarrow{\infty}$. We notice that for each level $m \in \mathbb{N}$ the wave number increases. Thus, conforming Galerkin FEM methods applied to this system of Helmholtz equations will encounter the so called pollution effect~\cite{BABUSKA2000, BABUSKA1995325}. It is well known that the pollution effect can be countered with an $hp$ adaptive FEM in the coefficient-constant case~\cite{Spence2023}. However, we face the Helmholtz equation with variable coefficients. Moreover, in each \revision{iteration $m$} the solutions of the previous iteration \revision{$m-1$} are used. Therefore, adapting the mesh size between iterations is not possible without interpolating the intermediate solution(s) from the previous iteration which possibly introduces additional inaccuracies. Due to the structure of the iteration scheme, the pollution effect and numerical errors affect not only solutions for higher frequencies but all of them and they are amplified throughout the iterations due to the auto-convolution on the right hand side. \revision{Hence, for our approximation scheme the number of iterations $N$ and the number of nodes/maximum diameter of the elements have an impact on the calculated solution. But, in each iteration the solutions to the Helmholtz equations can be computed in parallel because they only rely on data from the previous iteration.} 

\revision{For our numerical computations we implement a conforming 2D FEM solver of first order with linear Lagrange elements for the Helmholtz equation, employing a uniform triangulation of the domain.} We apply it to the system described in~\eqref{eq:helmholtz:systems:iteration:scheme:second:mentioning} 
\MarginB{vii.}
to demonstrate the effectiveness of our iteration scheme. The implementation of the FEM solver and the iteration scheme~\eqref{eq:helmholtz:systems:iteration:scheme:second:mentioning} \revision{is available at~\url{https://github.com/dazedsheep/FEMHelmholtzSolver}~\cite{GITSolver}}. \revision{For the source we use a regularized Dirac function which reads}
\begin{equation}
    \label{eq:results:regalurized:dirac}
     \hat{h}(x) = \Tilde{\delta}(x) = 
    \begin{cases}
        \frac{1}{4 \zeta} \left( 1 + \cos\left(\pi \frac{||x - x_0||_{l^2}}{2 \zeta}\right) \right) & ||x-x_0||_{l^2} \leq 2\zeta, \\
        0 & \text{else},
    \end{cases}
\end{equation}
where $x_0$ denotes the point source location. We fix the excitation frequency and consider the excitation to be periodic. Therefore, our constructed source is indeed in $L^2(0,T;L^2(\Omega))$.

\MarginA{9.}
\MarginB{i.}
\revision{
We investigate the impact of the frequency and the thereby caused pollution effect on numerical solutions computed by~\eqref{eq:helmholtz:systems:iteration:scheme:second:mentioning} by the following simulation study. The computational setup is $\Omega = B_{0.05}(0) \subset \mathbb{R}^2$, $B/A = 5$ in the whole domain, as well as $\rho_0 = 1000$ kg/$\text{m}^3$, $c = 1450$ m/s (water), $\text{m}/\text{s}$, $\gamma = 1$, $b = 10^{-9}$  $\text{m}^2/\text{s}$, $\zeta = 0.01$ resembling a source with a diameter of $0.1$ with a pressure source strength of $||\hat{h}||_{L^2(\Omega)} = 25 \cdot 10^3$ with a peak pressure of approximately $66$ MPa and a frequency of $100$ kHz. The impedance coefficient for each Helmholtz equation is set to $\beta = \frac{1}{c}$ to obtain absorbing boundary conditions mimicking the Sommerfeld radiation condition (if the source is placed in the center) and the mesh diameter is denoted by $h$. We calculate the time dependent $L^2$ error $e_j(t)$ of solutions of consecutive iterations by 
\begin{equation}
    \label{numerical:study:error}
    e_j(t) := || p^{j-1}(t, \cdot) - p^j(t, \cdot)||_{L^2(\Omega)},
\end{equation}
for $2 \leq j \leq N$, where $p^j(t,x)$ is calculated using~\eqref{pNansatz} and set the number of maximum iterations $N$ to $35$. By the choice of the nonlinearity being constant in the whole domain, the space dependent wavenumber $\kappa(x)$ is constant for each level $m$. We compute five solutions for different mesh parameters, starting with a mesh diameter of $h_{\text{FEM}}=0.00125$, i.e., 6141 nodes/degrees of freedom (DOF) and halving the mesh parameter after the computation of a solution. Figure~\ref{fig:pollution:effect} shows the $L^2$ error $e_j(0)$, $2 \leq j \leq 35$, for every iteration of each solution for different values of the element size. The highest wavenumber that we encounter in this experiment is approximately $298$. For each solution, we indicate where the pollution effect occurs noticeably. It is well known that $h_{\text{FEM}} |\kappa| \lesssim 1$ is a criterion such that solutions do not exhibit the pollution effect~\cite{BABUSKA1995325, Sauter2011}. This is confirmed by our simulations. However, once we encounter the pollution effect the solution blows up due to the errors introduced. These errors get amplified in consecutive iterations because of the autoconvolution of the previous solution(s). Hence, one may use $h_{\text{FEM}} |\kappa_m| \gtrsim 1$, for $1\leq m \leq N$ or in the case of a space dependent wavenumber one may use $h_{\text{FEM}} ||\kappa_m||_{L^\infty(\Omega)} \gtrsim 1$ as a criterion for selecting the the maximum number of iterations.
\begin{figure}
    \centering
    \includegraphics[width=.8\linewidth]{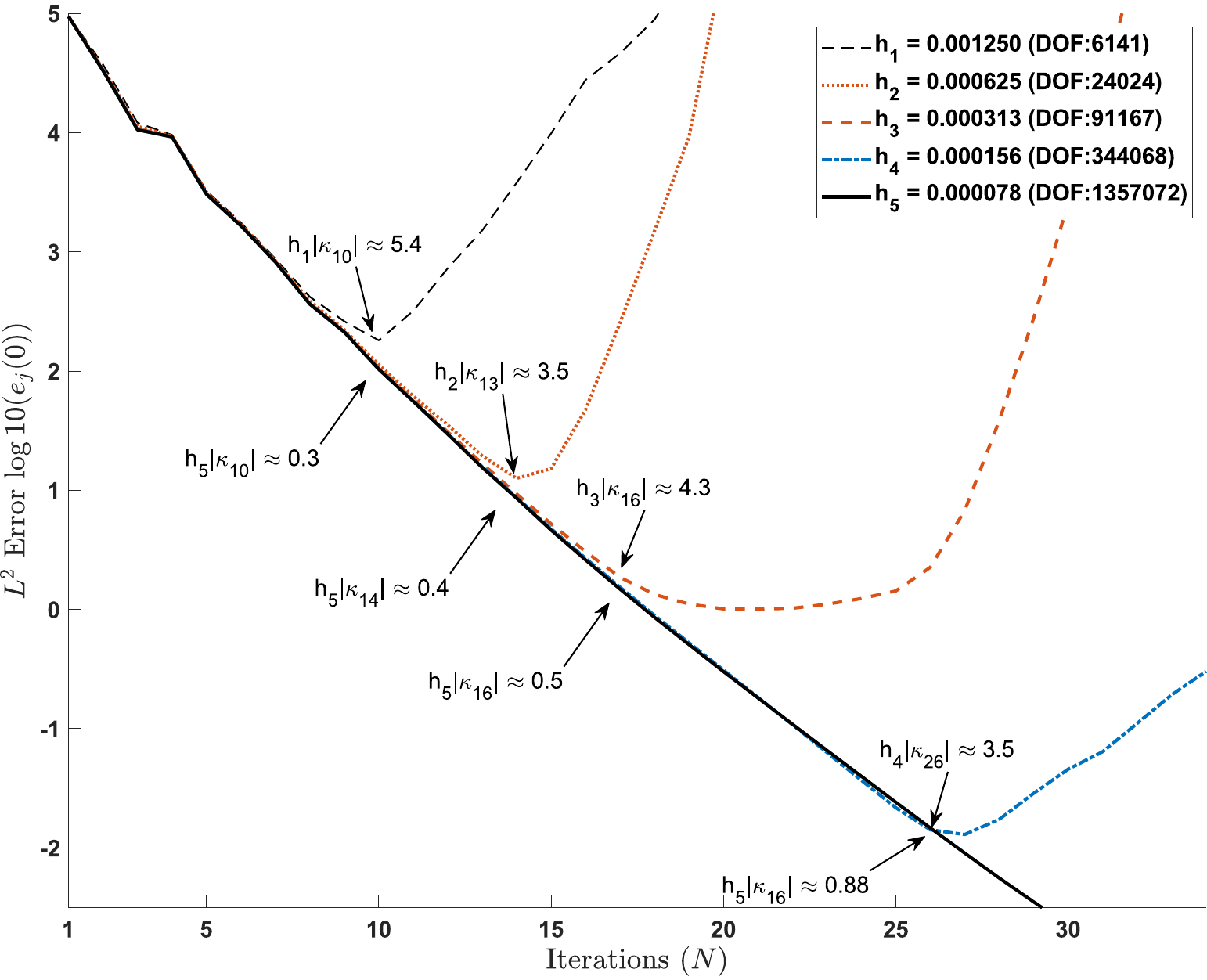}
	\caption{\revision{Effect of the mesh parameter $h$ and iterations $N$ on the error $e_j(0)$ of the solution, showcasing the impact of the pollution effect.}}
    \label{fig:pollution:effect}
\end{figure}
}

\MarginA{10.}
\MarginB{ii.}
\revision{
Theorem~\ref{thm:westervelt:nonlinear:solution:existence:uniqueness} states that a unique (weak) solution to the nonlinear periodic Westervelt equation only exists if the $L^2$-norm of the excitation is smaller than some threshold $\delta > 0$, which depends on the domain, the speed of sound, the spatially dependent nonlinearity $\eta$ and the boundary parameters. It is important to emphasize that $\eta$ itself depends on the speed of sound, and that the non-degeneracy condition $1 - 2\eta(x)p(t,x) > 0$ holds almost everywhere, cf. \eqref{degeneracy}. Consequently, sufficiently high pressure can lead to degeneracy. Thus, our next objective is to investigate the impact of $||\hat{h}||_{L^2(\Omega)}$ on the solution. To this end, we fix all other parameters as described above and vary the input pressure and excitation frequency for a fixed speed of sound. The mesh size is set to $5 \cdot 10^{-4}$ (DOF 35439), and the number of iterations $N$ is fixed at 30, ensuring that $h |\kappa_N| \ll 1$. To ensure that increased pressure  source strength does not introduce pollution effects, we verify selected cases using finer mesh resolutions.}

\revision{Figure~\ref{fig:saturation:effect:pressure:low:c} presents the results for $c = 50$ and three excitation frequencies: 100, 200, and 300 Hz. A linear fit based on the low-pressure solutions is included to demonstrate that, under low-pressure conditions, the $L^2$-norm of the solution increases linearly with $||\hat{h}||_{L^2(\Omega)}$. These trends are depicted by lines corresponding to each excitation frequency. The degeneracy condition is indicated as well. As long as $||\eta p(0,\cdot)||_\infty < 1/2$, the numerical solutions align with theoretical predictions. For small pressure source strenghts, the $L^2$-norm of the solution exhibits linear growth until saturation effects become significant, at which point nonlinear attenuation alters the behavior~\cite{PanfilovaSlounWijkstraSapozhnikovMischi:2021}. This manifests as a redistribution of energy from lower to higher harmonics. Once saturation is reached, the system begins to degenerate, ultimately leading to a blow-up of the solution. Figure~\ref{fig:saturation:effect:pressure:high:c} shows the results for higher sound speed $c=1450$ m/s and frequencies: 5, 7.5, and 10 kHz. The effects of saturation and nonlinear attenuation are clearly visible and are highlighted by comparing a linear fit -- based on low-energy solutions -- to a polynomial fit that captures the nonlinear behavior at higher pressures. The transfer of energy to higher harmonics with increasing pressure  source strength is further illustrated in Figure~\ref{fig:saturation:effect:pressure:solutions}. These findings are also of practical significance. For a sound speed of $1450$ m/s, the peak pressure emitted by the source ranges from $1.3$ MPa to $4$ MPa at $10 kHz$, from $133$ kPa to $1.5$ MPa at $7.5$ kHz and from $26$ kPa to $77$ kPa at $5$ kHz.  In all of these scenarios, degeneracy is not observed, indicating that the excitation remains below the threshold $\delta$ specified in Theorem~\ref{thm:westervelt:nonlinear:solution:existence:uniqueness}. 
\begin{figure}
    \centering
    \includegraphics[width=.7\linewidth]{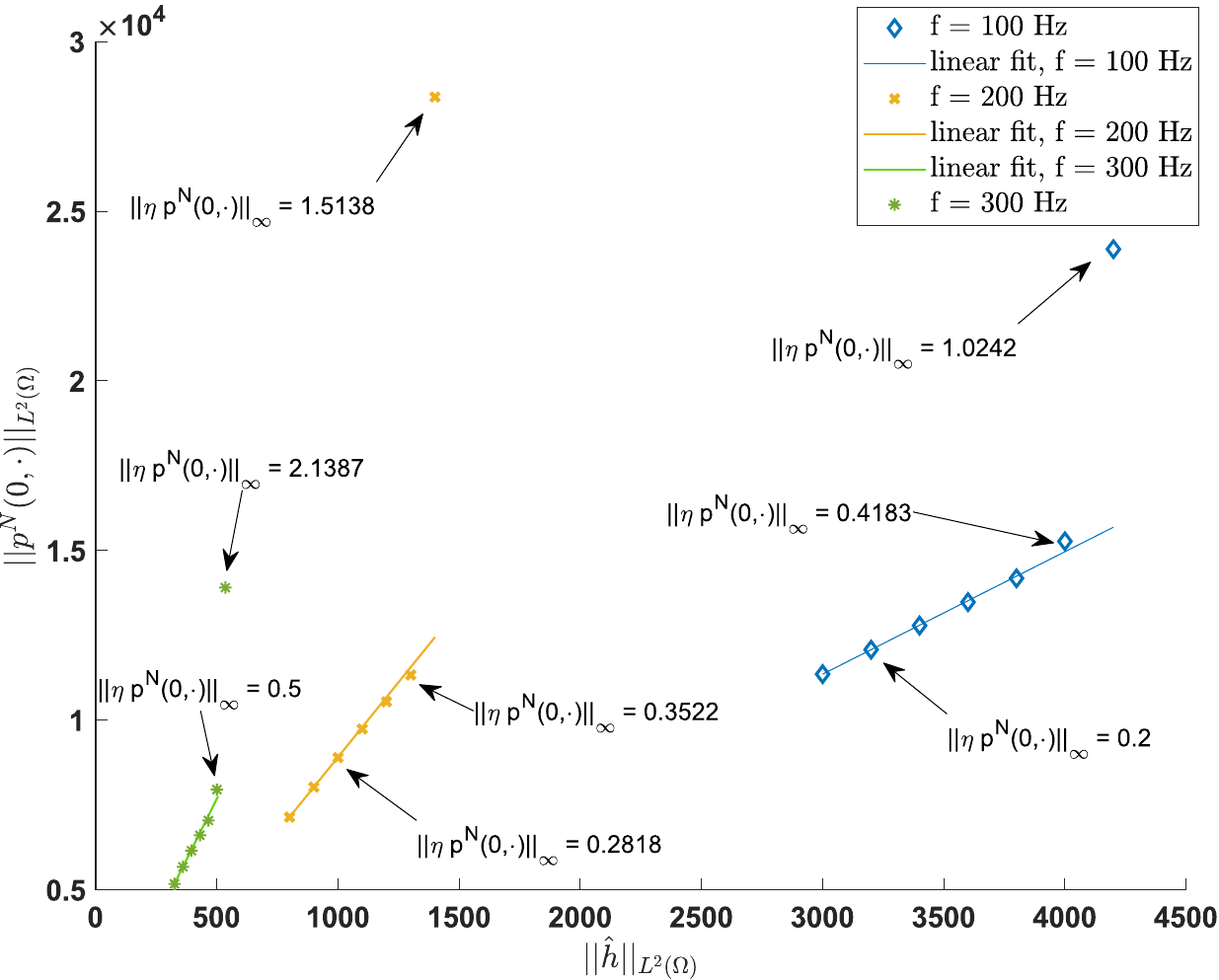}
	\caption{\revision{Effect of the pressure source strength on the solution for $c=50$ m/s for the frequencies 100, 200, and 300 Hz and a $B/A$ value of 5 in the whole domain.}}
    \label{fig:saturation:effect:pressure:low:c}
\end{figure}
\begin{figure}
    \centering
    \includegraphics[width=.7\linewidth]{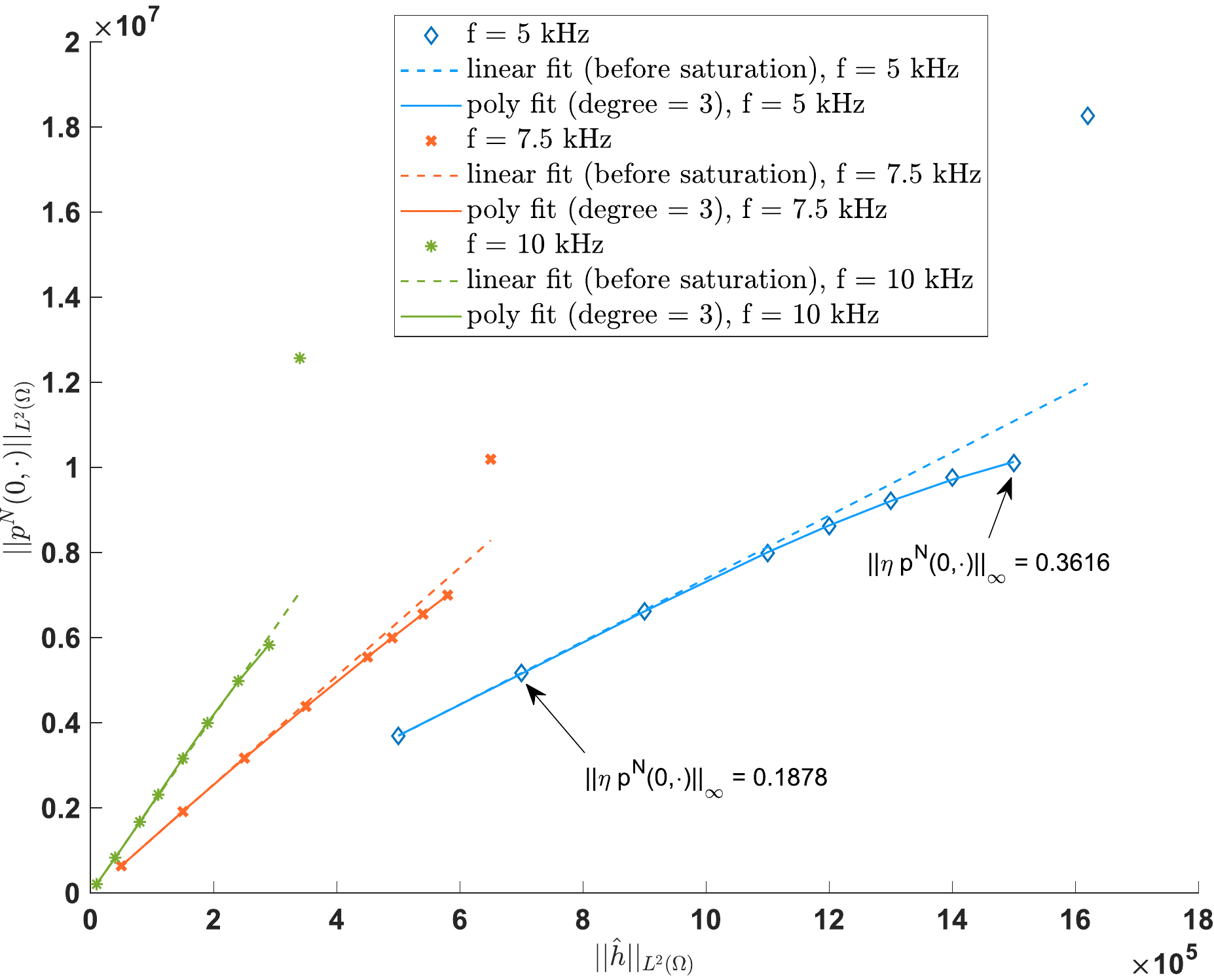}
	\caption{\revision{Effect of the pressure source strength on the solution for $c=1450$ m/s for the frequencies 5, 7.5, and 10 kHz and a $B/A$ value of 5 in the whole domain. The saturation effect appears before the solution blows up and leads to a shift of energy from lower harmonics to higher ones, hence, leading to a reduced $L^2$-norm of the computed solution.}}
    \label{fig:saturation:effect:pressure:high:c}
\end{figure}
\begin{figure}
    \centering
    \includegraphics[width=.62\linewidth]{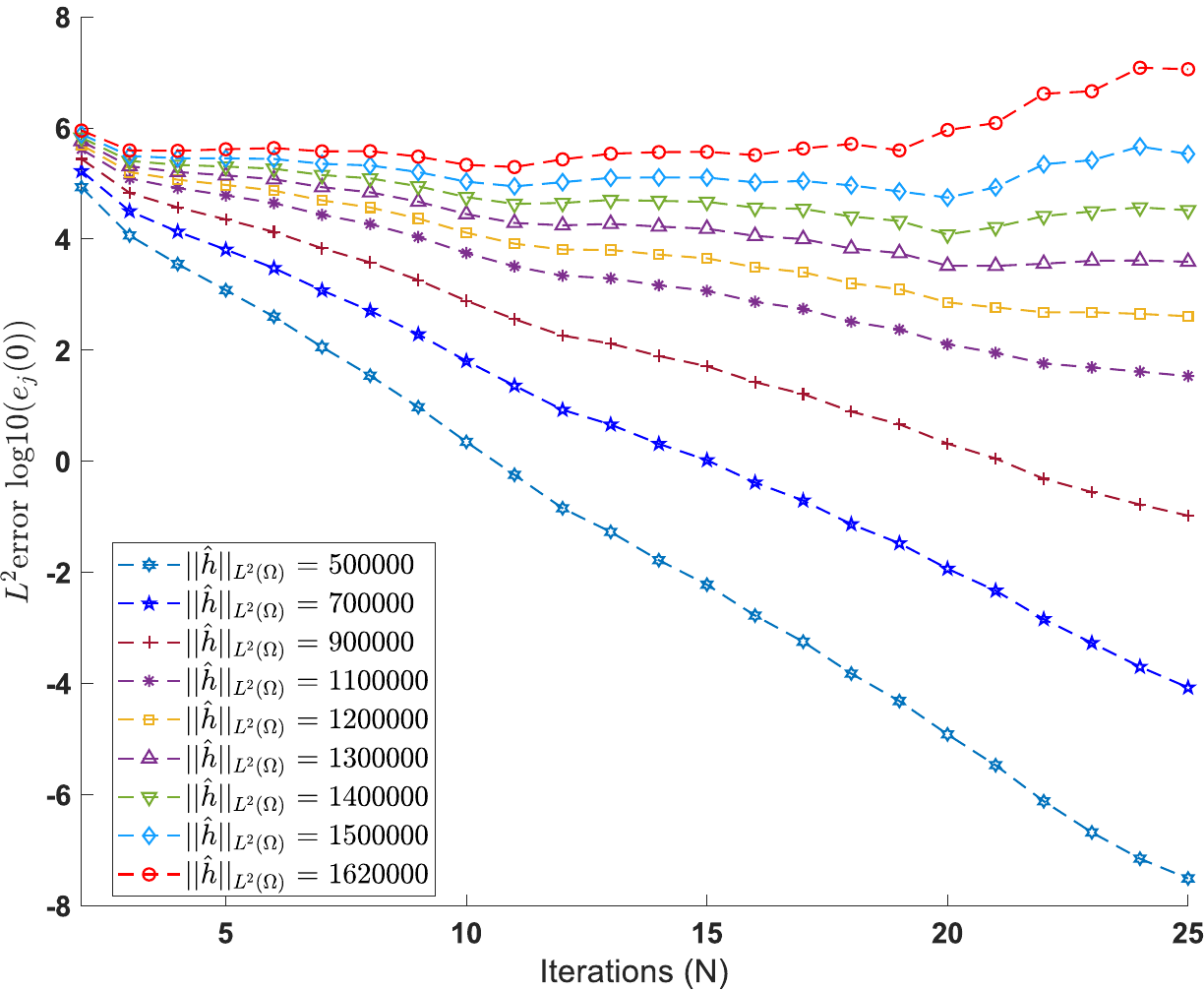}
	\caption{\revision{$L^2$ error of consecutive numerical solutions for different pressure source strenghts, $c=1450$ m/s for $10$ kHz and a $B/A$ value of 5 in the whole domain.}}
    \label{fig:saturation:effect:pressure:solutions}
\end{figure}
}

\MarginA{10.}
\revision{We run the simulations on an Intel Core i7 $13700$K equipped with 32 GB RAM. Our solver is implemented with MATLAB. As mentioned above, the iterative approximation scheme can be easily parallelized and  we apply the built-in \textit{parfor} for this purpose, using a parallel pool with 16 threads. Figure~\ref{fig:runtime} shows the runtime for the same configuration as above for $||\hat{h}||_{L^2(\Omega)} = 25 \cdot 10^3$ and a frequency of $5$~kHz for different element sizes $h_{\text{FEM}}$ and number of iterations.  We report the execution time for $N=30$ with respect to the element size $h_{\text{FEM}}$, number of nodes and elements in Table~\ref{tab:runtime}.
\begin{figure}
    \centering
    \includegraphics[width=.7\linewidth]{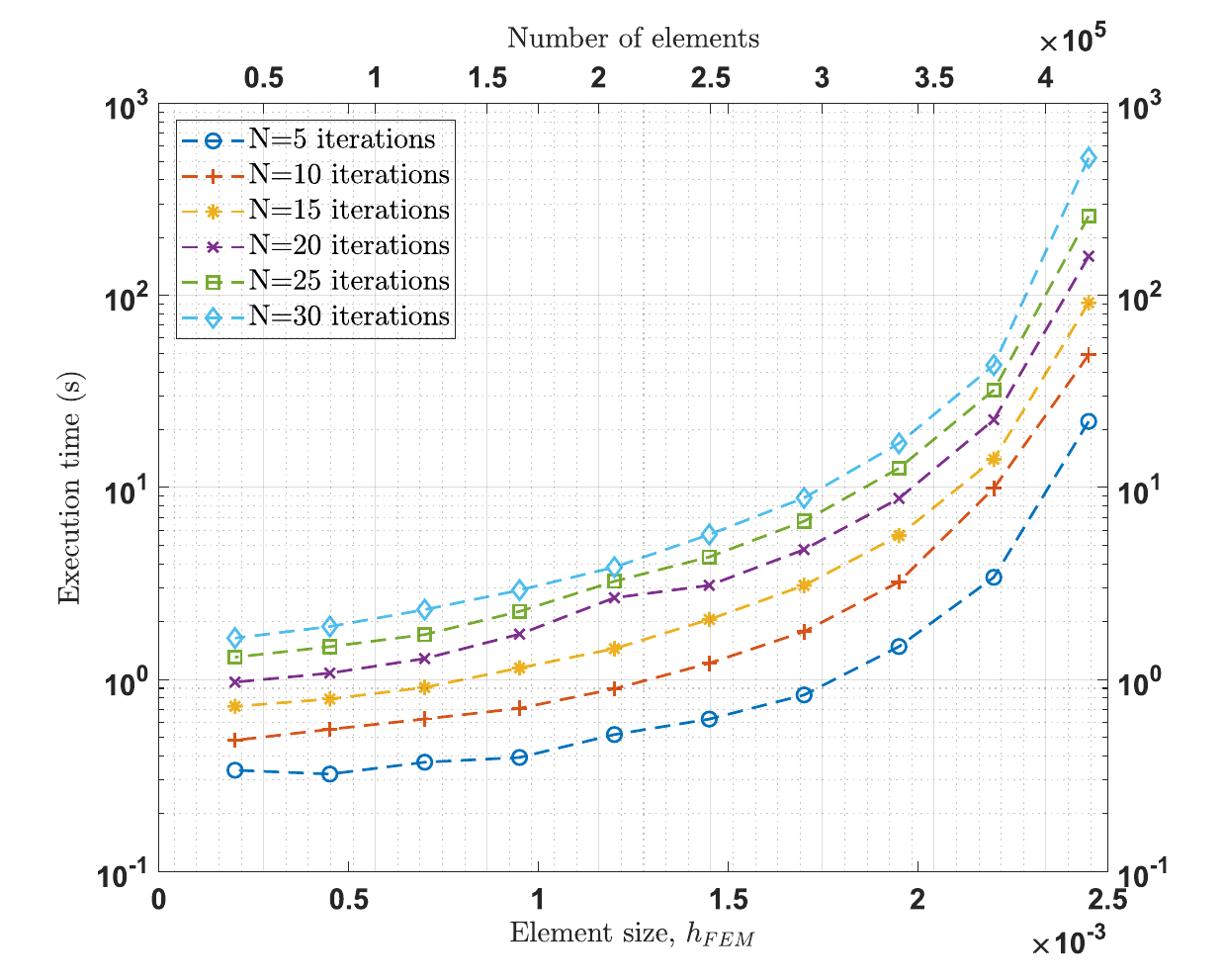}
	\caption{\revision{Execution time of the iterative approximation scheme for different element sizes and number of iterations.}}
    \label{fig:runtime}
\end{figure}
\begin{table}[]
\begin{tabular}{|l|c|c|c|}
\hline
\multicolumn{1}{|c|}{$h_{\text{FEM}}$} & Number of nodes & \multicolumn{1}{l|}{Number of elements} & \multicolumn{1}{l|}{Execution time (s)} \\ \hline
0.0022                                 & 2075            & 4003                                    & 1.8876                                  \\ \hline
0.0017                                 & 3369            & 6551                                    & 2.9202                                  \\ \hline
0.0012                                 & 6733            & 13203                                   & 5.6883                                  \\ \hline
0.0007                                 & 19375           & 38299                                   & 16.9757                                 \\ \hline
0.0002                                 & 214592          & 427609                                  & 521.4009                                \\ \hline
\end{tabular}
\caption{\revision{Execution time for iterative approximation scheme~\eqref{eq:helmholtz:systems:iteration:scheme:second:mentioning} for $N=30$ iterations, domain $B_{0.05}(0)$, $B/A=5$ in the whole domain and pressure source of $25 \cdot 10^3$ at a frequency of $5$~kHz.}}
\label{tab:runtime}
\end{table}
}

\revision{Next, we show two examples which are of practical use.} The domain of propagation is selected to be $ \Omega = B_{0.2}(0) \subset \mathbb{R}^2$ and a monopole source is placed at $(0,0)^T$. The parameters used in the simulations are chosen as typical for water, in particular $B/A = 5$~\cite{PANTEA20131012, STURTEVANT2012}, $\rho_0 = 1000$ kg/$\text{m}^3$, $\omega = 2\pi \frac{1}{10^{-5}}$ $\text{rad}/\text{s}$, $b = 10^{-9}$ $\text{m}^2/\text{s}$ , mesh parameter $0.0005$ m, $c = 1450$ $\text{m}/\text{s}$, $\gamma = 1$, $\zeta = 0.0005$, sound pressure $0.5$ MPa, and number of iterations $15$. We further set the impedance coefficient for each Helmholtz equation to $\beta = \frac{1}{c}$ to obtain absorbing boundary conditions mimicking the Sommerfeld radiation condition (if the source is placed in the center). 

Figure~\ref{fig:study2:result:nonlinear:vs:linear:100khz} depicts the acoustic pressure along the line $(-0.2,0)^T$ to $(0.2,0)^T$ where the source is located at $(0,0)^T$. Due to attenuation, the further away we move from the source, the greater the impact of the harmonics introduced by the nonlinearity parameter is on the waveform. Figure~\ref{fig:study2:100khz:freq} illustrates the frequency components present at different spatial locations, showing that the first three harmonics remain above $-40$ dB. Figure~\ref{fig:frequency_scenario1_impact_nonlinearity} presents the relative pressure at different locations in the domain. In Figure~\ref{fig:freq_in_source_sc1}, which corresponds to the source location, the excitation frequency is more than $20$ dB stronger than other frequency components, as expected. This is also evident in Figure~\ref{fig:study2:result:nonlinear:vs:linear:100khz} at $x=0$, where the solution with a nonzero nonlinearity parameter closely matches the linear case. However, this changes rapidly as we move away from the source. At $(0.05,0)^T$, cf. Figure~\ref{fig:freq_near_source_sc1}, the contribution of the excitation frequency is significantly attenuated, and this effect is also visible in Figure~\ref{fig:study2:result:nonlinear:vs:linear:100khz} at $x=0.05$. Further away from the source, at $(0.2,0)^T$, the relative pressure of the excitation frequency becomes comparable to that of the first two harmonics. In Figure~\ref{fig:study2:100khz:time}, we plot the acoustic pressure over time in a single point within $\Omega$. Comparing the nonlinear case to the linear one, the non-uniform phase velocity is clearly visible in the nonlinear case. 
\begin{figure}
    \centering
    \includegraphics[width=.5\linewidth]{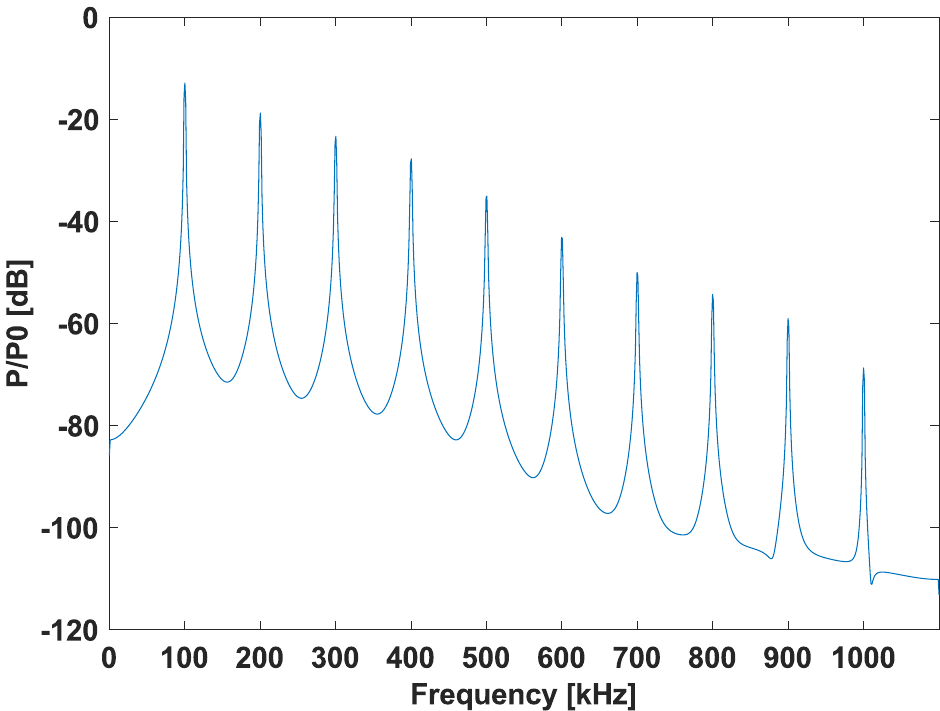}
		\caption{Frequency components of $p(t,x)$ at $(0,-0.1)^T$, where we applied a Hann window~\cite{oppenheim1999discrete} in time.}
		\label{fig:study2:100khz:freq}
\end{figure}
\begin{figure}[h]
	\centering
	\includegraphics[width=1.0\textwidth]{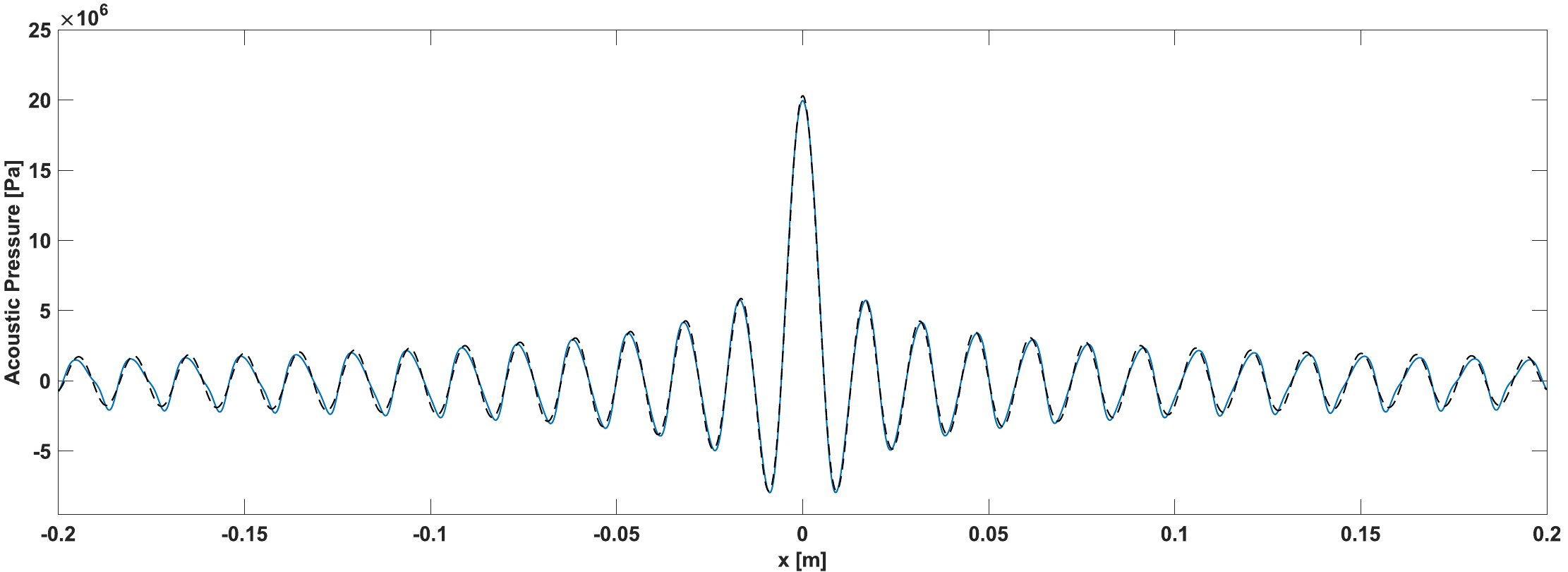}
	\caption{The propagation of the wave along a line originating at $(-0.2,0)^T$ and ending at $(0.2,0)^T$. The black dotted line denotes the linear case where 
    $\eta = 0$ 
    at $t=0$. The effect of nonlinearity on the waveform becomes more pronounced the farther one moves from the source.}
	\label{fig:study2:result:nonlinear:vs:linear:100khz}
	\centering
\end{figure}
\begin{figure}
\centering
    \begin{subfigure}{0.3\textwidth}
        \includegraphics[width=1.0\linewidth]{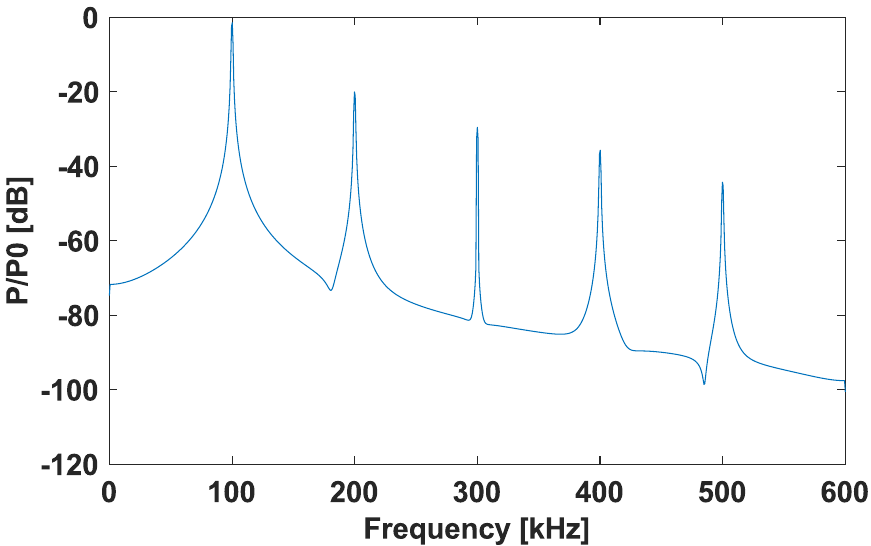} 
        \caption{Relative pressure of frequency components at $(0,0)^T$.}
        \label{fig:freq_in_source_sc1}
        \end{subfigure}
    \begin{subfigure}{0.3\textwidth}
        \includegraphics[width=1.0\linewidth]{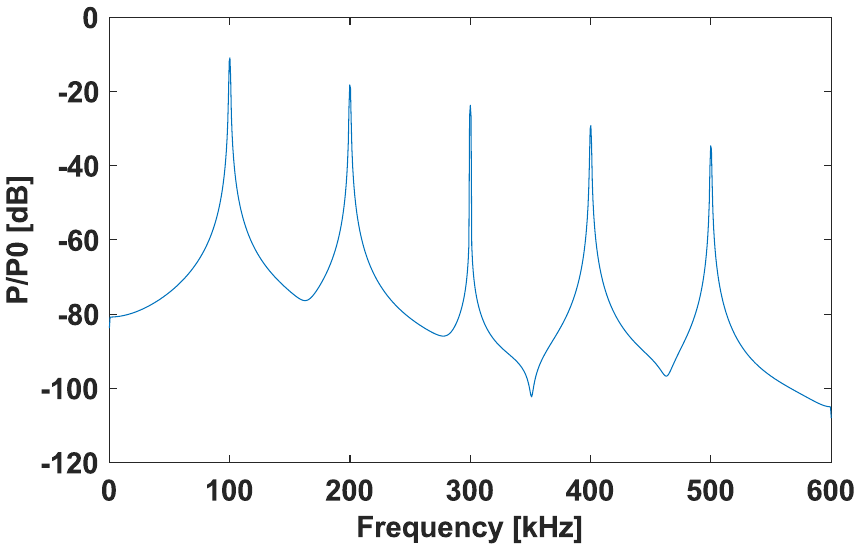} 
        \caption{Relative pressure of frequency components at $(0.05,0)^T$.}
        \label{fig:freq_near_source_sc1}
        \end{subfigure}
    \begin{subfigure}{0.3\textwidth}
        \includegraphics[width=1.0\linewidth]{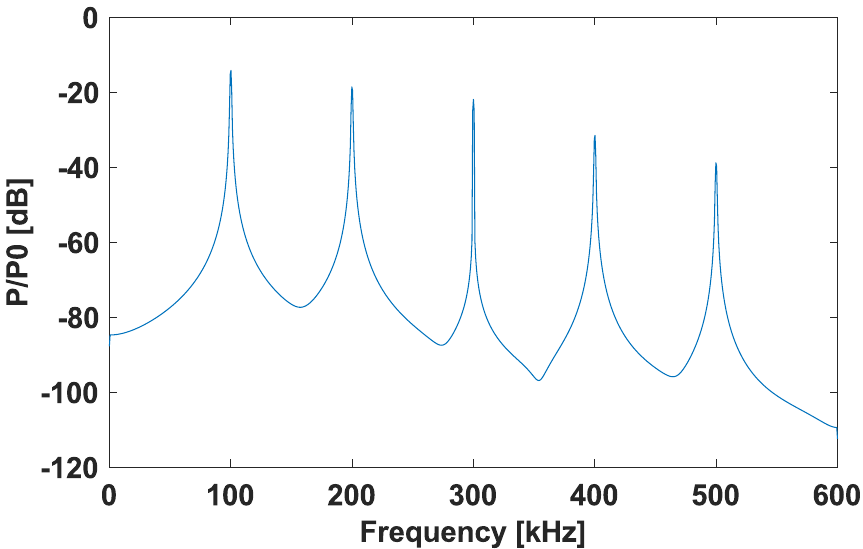} 
        \caption{Relative pressure of frequency components at $(0.2,0)^T$.}
        \label{fig:freq_far_from_source_sc1}
        \end{subfigure}
    \caption{Relative pressure of frequency components in different locations of the domain, $P_0$ denotes the maximum pressure in the domain.}
    \label{fig:frequency_scenario1_impact_nonlinearity}
\end{figure}
\begin{figure}
    \centering
	\includegraphics[width=0.7\textwidth]{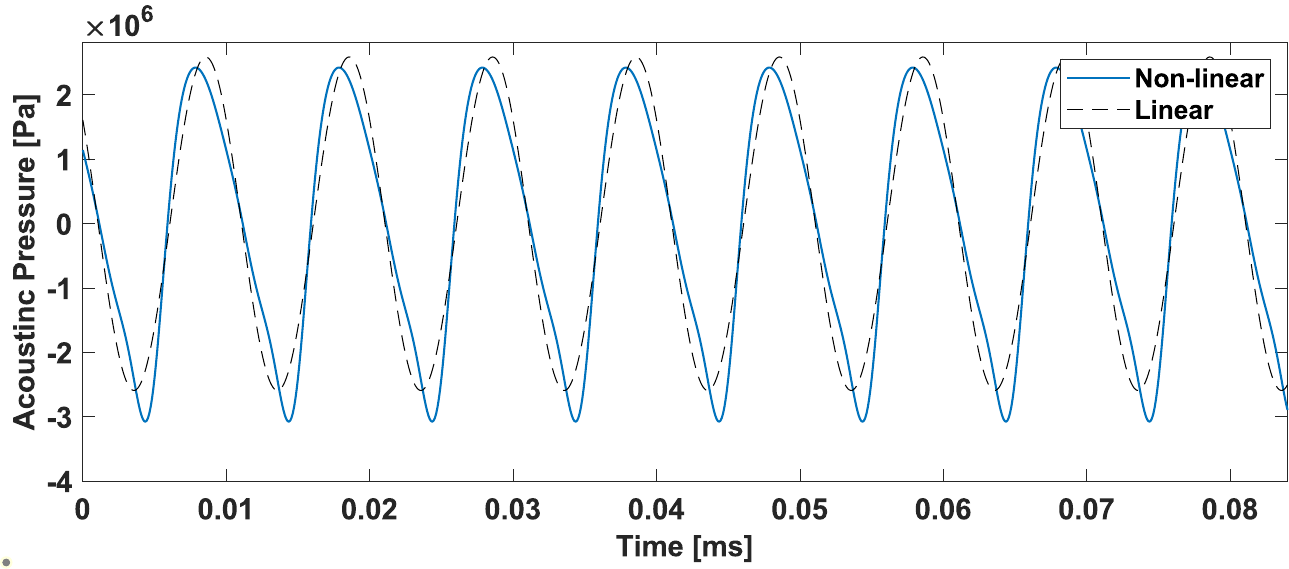}
	\caption{Evolution over time of the acoustic pressure at $(0,-0.1)^T$.}
	\label{fig:study2:100khz:time}
\end{figure}

It is well known that if the wave vectors are not perpendicular to the tangent in boundary points the employed absorbing boundary conditions cause spurious reflections~\cite{SHEVCHENKO2012461}. This is avoided in  case of a source centered in a circular domain, where the wave hits the boundary in normal direction. This is not necessarily the case for other source positions relative to the boundary and needs to be taken into account accordingly.
A self-adaptive boundary condition scheme for the quasilinear Westervelt equation in time domain has been introduced in~\cite{MUHR2019279}. The absorbing boundary conditions are adapted in~\textit{real-time} taking into account incidence wave vectors. Here, estimates of the solution computed in time step $t_{n-1}$ are used to adapt the boundary conditions in the next time step $t_n$. This works remarkably well in space-time finite element methods. However, we are facing~\eqref{eq:helmholtz:systems:iteration:scheme:second:mentioning}, a system of Helmholtz equations. Thus, adaptation over time is not an option.
If there is only a single source or if $\nabla p$ is known, the gradient method can be employed~\cite{GORDON2015295}. In our iteration scheme, each iteration of~\eqref{eq:helmholtz:systems:iteration:scheme:second:mentioning} virtually adds new \textit{sources} with different frequencies, though. A detailed investigation and the derivation of adaptive absorbing boundary conditions for~\eqref{eq:helmholtz:systems:iteration:scheme:second:mentioning} is beyond the scope of this paper.

Next, we turn to a scenario that resembles a watershot experiment (cf. Figure~\ref{fig:study:scenario}). Here, we are interested in boundary measurements of a domain filled with water. This type of experiment is of high interest for evaluating the accuracy of inverse methods~\cite{kaltenbacher2023nonlinearityparameterimagingfrequency} for the reconstruction of the nonlinearity parameter. We consider a circular air domain with a radius of one meter and a concentrically placed circular water tank with a radius of $0.35$ meter including three phantoms. The phantoms differ in size and material. Their radii are $0.05, 0.04$ and $0.06$ meters, respectively. \revision{We set the element size to $h_\text{FEM} = 0.0008$, hence avoiding the pollution effect.} The sound speed is set accordingly, that is,  $344$ m/s for air and $1450$ m/s for water. We use the following $B/A$ values which determine the nonlinearity parameter $\eta$: air (1), water (5), first phantom  (9), second phantom (10), third phantom (12). A linear array with nine elements spaced by $\lambda / 4$ in horizontal direction, is placed in the center of the domain. Each element of the linear array emits a $10$ kHz sinusoidal ultrasound wave with a peak pressure of \revision{$10$~MPa with $\zeta = 0.002$, cf.~\eqref{eq:results:regalurized:dirac} and Figure~\ref{fig:linear_array_sources}. All sources have the same angular component. Thus, this design of the sources leads to a superposition of the emitted wave into the vertical directions.}\MarginB{ix.} We use~\eqref{eq:helmholtz:systems:iteration:scheme:second:mentioning} to compute the first six harmonic solutions which are then used to approximate $p(t,x)$. Figure~\ref{fig:study:approximate:pressure:profile} depicts the approximation of $p(0,x)$ with all three phantoms being present. The effect of the nonlinearity parameter and the different propagation speeds is clearly visible. 
In view of the imaging task of reconstructing inclusions from boundary observations, we are interested in boundary measurements of the domain filled with water for the cases of no phantom, one phantom, two phantoms and three phantoms. To this end, $p_j(t,x)$, for $j \in \{0,1,2,3\}$ denotes the solution with no phantoms, phantom one, phantom two, or phantom three present. 
Solutions with combinations of phantoms are denoted by the corresponding indices, i.e., $p_{1,2}(t,x)$ with phantoms one and two and $p_{1,2,3}(t,x)$ with all three phantoms present. Figure~\ref{fig:study:pressure:on:boundary} displays the boundary values relative to the location of the source on the surface of the water tank. In the left column we show the difference between the solutions without phantoms and single phantoms present. The right column displays the solution without phantoms, and the difference to having two (the first and second) and three phantoms present. The first and second phantom have similar $B/A$ values but differ in size and position. This is reflected in the first two figures in the left column of Figure~\ref{fig:study:pressure:on:boundary}. The third phantom has a $B/A$ value of $12$ and lies closest to the source; this already causes the boundary values to differ by more than $1$ kPa. 
\begin{figure}
\centering
    \begin{subfigure}{0.45\textwidth}
        \centering
        \includegraphics[width=1.0\linewidth]{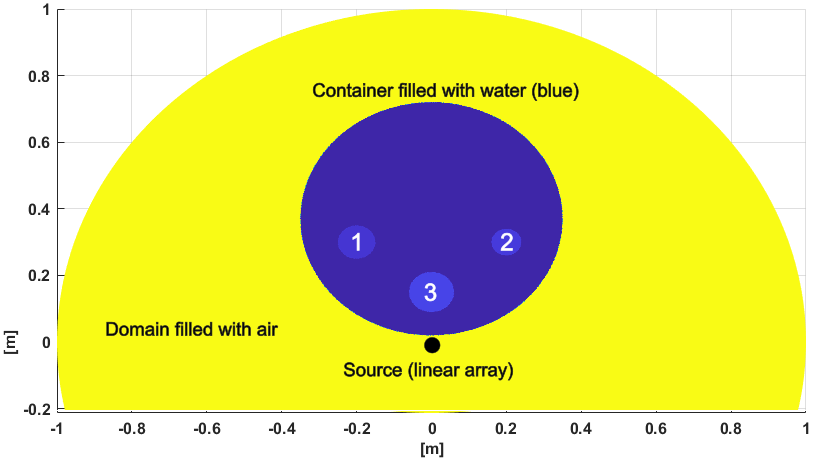} 
        \caption{\revision{Container filled with water and three phantoms inside, placed in a circular domain filled with air.}}
        \label{fig:study:scenario}
        \end{subfigure}
    \begin{subfigure}{0.45\textwidth}
        \centering
        \includegraphics[width=0.9\linewidth]{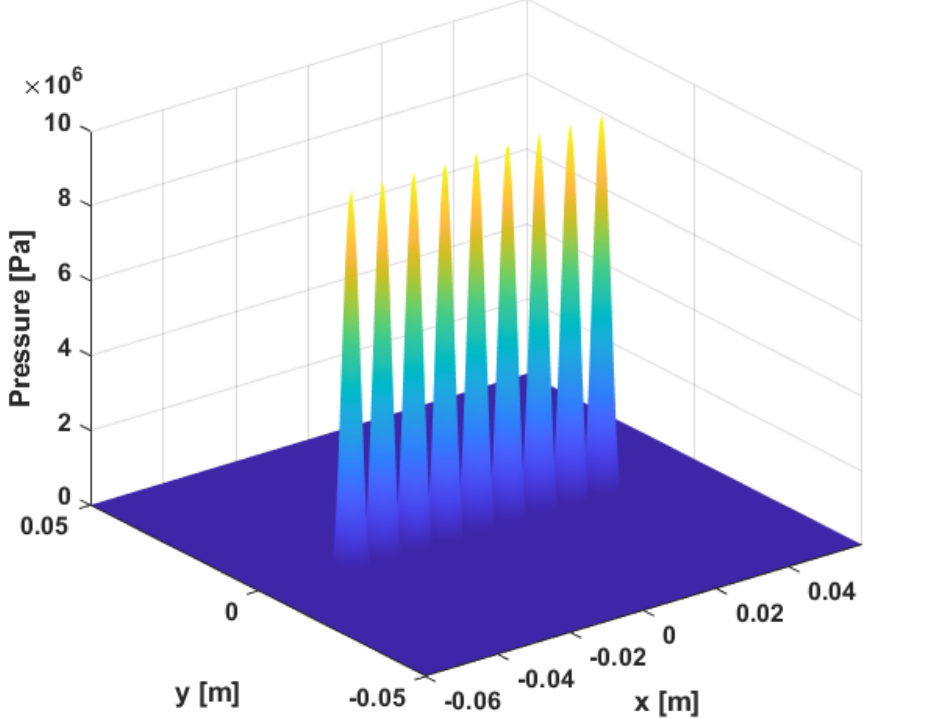} 
        \caption{\revision{The sources arranged in a linear array with $\lambda / 4 $ spacing.}}
        \label{fig:linear_array_sources}
        \end{subfigure}
    \caption{\revision{The domain for the watershot experiment with three different phantoms (left), and the sources arranged in a linear array (right).}}
    \label{fig:simulation_water_shot_setup}
\end{figure}

\begin{figure}[h]
	\centering
	\includegraphics[width=1.0\textwidth]{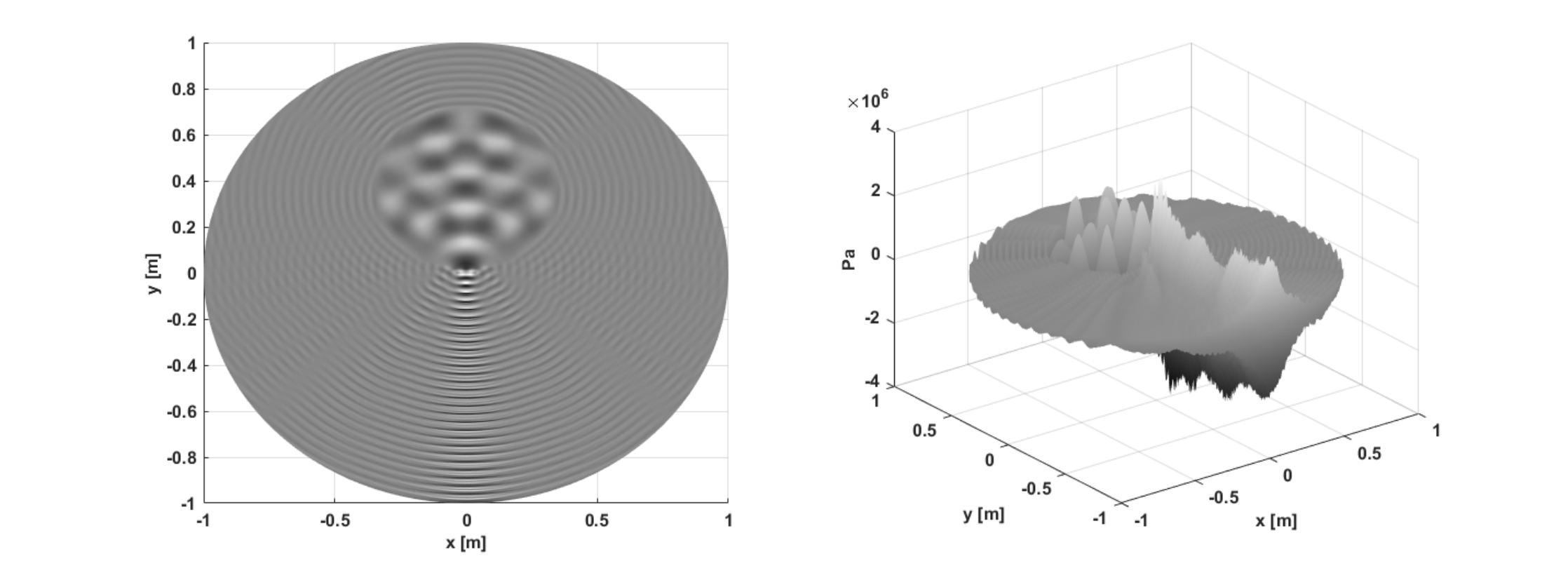}
	\caption{Simulation result of $p_{1,2,3}(0,x)$, left x-y view of $p_{1,2,3}(0,x)$ and right a 3D plot of $p_{1,2,3}(0,x)$.}
	\label{fig:study:approximate:pressure:profile}
	\centering
\end{figure}
\begin{figure}[h]
	\centering
	\includegraphics[width=0.9\textwidth]{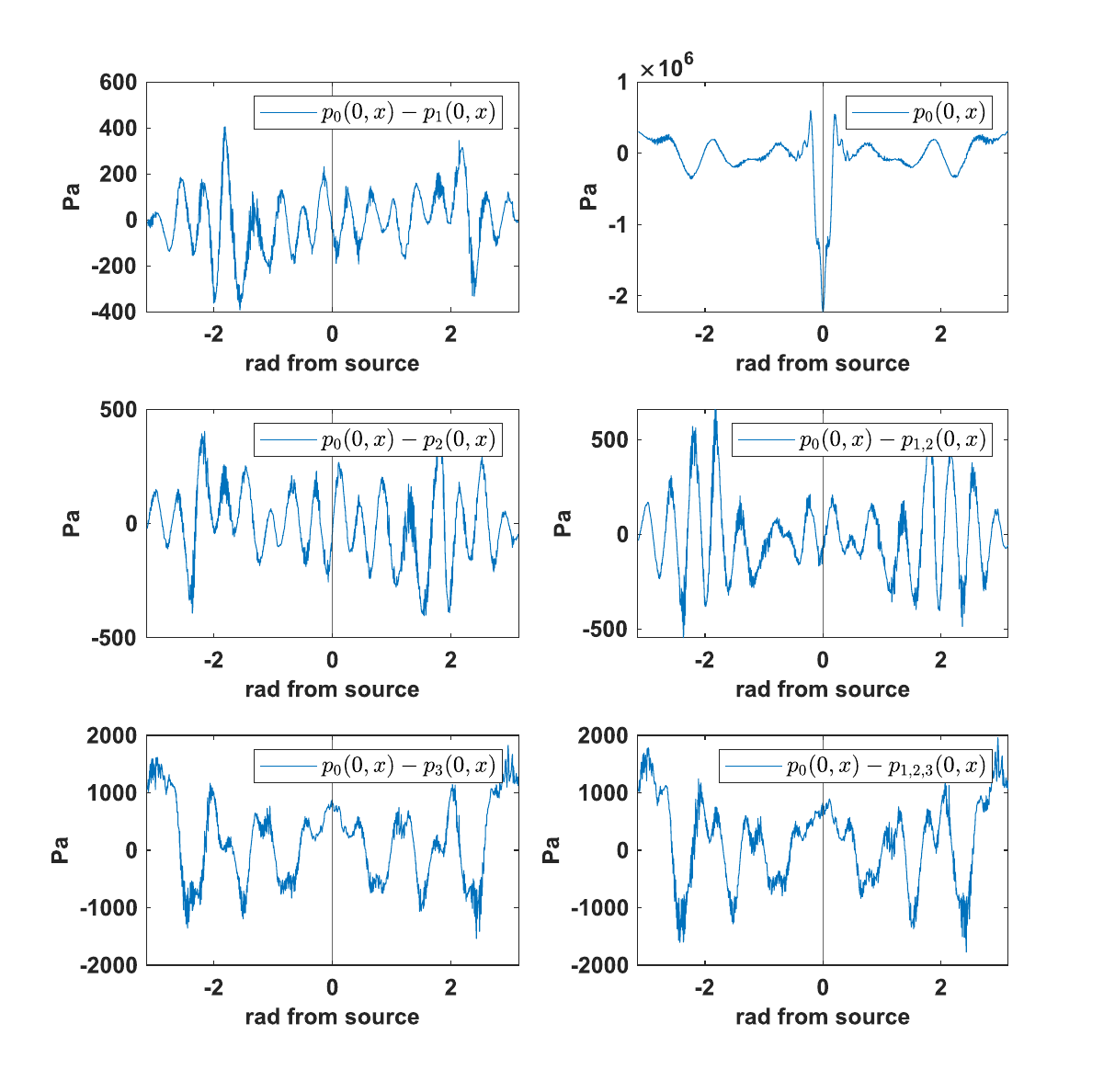}
	\caption{Boundary values of the container filled with water with different phantoms present. On the left side we show the difference between $p_0(0,x)$ and $p_j(0,x)$, $j \in \{1,2,3\}$, the right side of the figure shows the difference between $p_0(0,x)$ and the solutions including two, and three phantoms. The x-axis denotes the angle in polar coordinates, with the source being located at zero radian.}
	\label{fig:study:pressure:on:boundary}
	\centering
\end{figure}

\section{Conclusions}
\label{sec:conclusions}
In conclusion, our study aims to provide analytic and numerical tools for ultrasound imaging, particularly focusing on the nonlinear periodic Westervelt equation. This equation is pivotal in understanding tissue interactions and discrimination in ultrasound diagnostics. Through our investigation, we have shown that solutions to the nonlinear periodic Westervelt equation exist under certain smallness assumptions on the excitation. Additionally, we have introduced an iterative scheme that approximates solutions in frequency domain. 

In this work we considered an excitation in the interior of the domain, which has immediate applications in invasive ultrasound, such as intravascular ultrasound imaging.
Also experimental settings in which ultrasound is generated by a transducer array located outside the object to be imaged can be described this way 
-- e.g., with the source and imaging object immersed in a water tank. However, in clinical applications one envisages ultrasound  basically just propagating through the body, with the transducer array being attached to its surface, thus the source is located on a part of the boundary of the computational domain $\Omega$. Therefore, our future research will include investigations on the situation where the excitation is localized on $\Gamma \subset \partial\Omega$.

Also the development and use of enhanced absorbing boundary conditions~\cite{engquist1977absorbing, GIVOLI2004}  for a nonlinear wave equation in frequency domain appears to be a relevant topic of research in our context. Based on the results obtained here, we plan to provide precise and efficient reconstruction schemes for nonlinear ultrasound tomography, which amounts to reconstructing space dependent coefficients $c(x)$ and/or $\eta(x)$ in \eqref{eq:westervelt}.

\section*{Acknowledgments}
This research was funded in part by the Austrian Science Fund (FWF) 
[10.55776/P36318]. 
For open access purposes, the author has applied a CC BY public copyright license to any author accepted manuscript version arising from this submission

\revision{The authors thank the anonymous reviewers for providing valuable comments and suggestions that have led to an improved version of the paper.}

\bibliographystyle{siamplain}
\bibliography{references}


\appendix
\section{Proof of Theorem~\ref{thm:westervelt:linear:solution:existence:uniqueness}}
\label{appendix:proof:thm:3:1}
First, we derive the weak form of~\eqref{eq:westervelt:linear:periodic}. To this end, we multiply by $-\frac{1}{\alpha}\Delta\psi$ for some  $\psi \in H^1(0,T;H^1(\Omega))\cap L^2(0,T;H^2(\Omega))$, and integrate over space time, using integration by parts in the first term (which we have augmented in order to be able to use the boundary conditions) 
\[
\begin{aligned}
&\int_0^T\int_\Omega (u_{tt}+\frac{\gamma}{\beta}u_t)(-\Delta\psi)\, dx\, dt = \int_0^T\Bigl(\int_\Omega \nabla(u_{tt}+\frac{\gamma}{\beta}u_t)\cdot\nabla\psi\, dx 
-\int_{\partial\Omega} (u_{tt}+\frac{\gamma}{\beta}u_t)\nabla\psi\cdot\vecn\, dx\Bigr) dt\\
&=\int_0^T\Bigl(\int_\Omega \Bigl(-\nabla u_t\cdot\nabla\psi_t+\frac{\gamma}{\beta}\nabla u_t\cdot\nabla\psi\Bigr)\, dx 
+\int_{\partial\Omega} \frac{1}{\beta}(\nabla u_t\cdot\vecn)\, (\nabla\psi\cdot\vecn)\, dS\Bigr) dt + \left[\int_\Omega \nabla u_t\cdot\nabla\psi\, dx \right]_0^T.
\end{aligned}
\]
Moreover, the compatibility condition 
\[
0=\int_0^T (\beta u_{t}+\gamma u+\nabla u\cdot\vecn)\,dt=\int_0^T (\gamma u+\nabla u\cdot\vecn)\,dt 
\]
(the last identity following from $u(T)=u(0)$) 
that will be  needed to concluded the boundary condition $\beta u_{t}+\gamma u+\nabla u\cdot\vecn=0$ from its time differentiated version, is taken care of by introducing an auxiliary function $h$ and imposing its periodicity.
This yields the variational formulation
\begin{equation}
    \label{eq:westervelt:linear:weak:form}
\begin{aligned}
&u\in X, \quad u(0)=u(T), \ u_t(0)=u_t(T), \ h(0)=h(T)\\
&\text{ and for all }\psi \in H^1(0,T;H^1(\Omega))\cap L^2(0,T;H^2(\Omega)), \ \zeta\in L^2(0,T;L^2(\partial\Omega)):\\
&\int_0^T\Bigl(\int_\Omega \Bigl(-\nabla u_t\cdot\nabla\psi_t+\frac{\gamma}{\beta}\nabla u_t\cdot\nabla\psi
+\frac{1}{\alpha}(c^2\Delta u+b\Delta u_t+(\frac{\gamma}{\beta}\alpha-\mu)u_t)\, \Delta\psi
\Bigr)\, dx \\
&\qquad+\int_{\partial\Omega} \frac{1}{\beta}(\nabla u_t\cdot\vecn)\, (\nabla\psi\cdot\vecn)\, dS
+\int_{\partial\Omega} (h_t-u_t-[\gamma+\vecn\cdot\nabla]u)\, \zeta\, dS
\Bigr) \,dt\\
&+\left[\int_\Omega \nabla u_t\cdot\nabla\psi\, dx \right]_0^T 
= -\int_0^T\int_\Omega  \frac{1}{\alpha} f \Delta\psi\, dx \,dt.
\end{aligned}
\end{equation}
Indeed, reversing the above integration by parts step in \eqref{eq:westervelt:linear:weak:form} yields
\begin{equation}
    \label{eq:westervelt:linear:weak:form_L2}
\begin{aligned}
&u\in X, \quad u(0)=u(T), \ u_t(0)=u_t(T), \ \int_0^T [\gamma+\vecn\cdot\nabla]u\,dt=0\\
&\text{ and for all }\psi \in H^1(0,T;H^1(\Omega))\cap L^2(0,T;H^2(\Omega)):\\
&\int_0^T\Bigl(\int_\Omega \frac{1}{\alpha}\Bigl(\alpha u_{tt}-c^2\Delta u-b\Delta u_t+\mu u_t-f\Bigr)\, (-\Delta\psi)\, dx + \int_{\partial\Omega} \frac{1}{\beta}(\beta u_{tt}+\gamma u_t+\nabla u_t\cdot\vecn)\, (\nabla\psi\cdot\vecn)\, dS\Bigr) \,dt
= 0.
\end{aligned}
\end{equation}

\smallskip

Second, we employ a Galerkin method. For this purpose, we utilize the fact that the eigenfunctions $(\phi_n)_{n\in \mathbb{N}}$ of the impedance Laplacian $-\Delta_\gamma$, that is, 
\[
-\Delta \phi_k = \lambda_k\phi_k \text{ in }\Omega, \qquad 
[\gamma+\vecn\cdot\nabla]\phi_k =0 
\text{ on }\partial\Omega,\]
form an orthogonal basis of $H^1(\Omega)$, an orthonormal basis of $L^2(\Omega)$ and  with $\tilde{V}_n := \text{lin}\{\phi_1, \ldots, \phi_n\}$, $\bigcup_{n\in \mathbb{N}} \tilde{V}_n$ is dense in $H^1(\Omega)$~\cite{evans2010}. 
As an ansatz space for $h$ we use the sequence $(\eta_k)_{k\in\mathbb{N}}$ defined by $\eta_k=\text{tr}_{\partial\Omega}\phi_k$ which is indeed linearly independent by construction and whose span is dense in $L^2(\partial\Omega)$, due to the trace theorem; this choice guarantees that the boundary trace of $u_{n\,t}+[\gamma+\vecn\cdot\nabla]u_n$ (which coincides with the boundary trace of $u_{n\,t}$ by construction of the basis functions $\phi_k$) 
can be used as a test function later on in the energy estimates.  
Thus altogether, we arrive at the Galerkin ansatz space $V_n := \text{lin}\{\phi_1, \ldots, \phi_n\}\times\text{lin}\{\eta_1, \ldots, \eta_n\}$ whose union $\bigcup_{n\in \mathbb{N}} V_n$ is dense in $H^1(\Omega)\times L^2(\partial\Omega)$.
Now, plugging the ansatz $(u_n(t,\cdot,),h_n(t, \cdot,)) := \sum_{k=1}^n( (\mathfrak{a}_k(t) \phi_k,\mathfrak{c}_k(t)\eta_k)\in V_n$ into~\eqref{eq:westervelt:linear:weak:form}
and testing with $\phi_j$, $\eta_j$, $j\in\{1,\ldots,n\}$
\footnote{Formally, after reversing the time integration by parts in the highest order term, we set $\psi(x,s)=\delta_t(s)\phi_j(x)$, $\zeta(x,s)=\delta_t(s)\eta_j(x)$; more precisely we set $\psi(x,s)=\chi_{[0,t]}(s)\phi_j(x)$, $\zeta(x,s)=\chi_{[0,t]}(s)\eta_j(x)$ and differentiate with respect to $t$ then.} 
yields the system of ODEs
\begin{equation}
    \label{eq:westervelt:linear:galerkin:form:full:ODE:system:a}
\begin{aligned}
    &   
\begin{bmatrix} 
I_{n \times n} & 0_{n \times n}& 0_{n \times n}\\ 
0_{n \times n} & \mathbf{H}     &0_{n \times n}\\
0_{n \times n} & 0_{n \times n}& \mathbf{D}\\ 
\end{bmatrix}
\dot{\vec{z}}(t) =  
\begin{bmatrix} 
0_{n \times n} & I_{n \times n}& 0_{n \times n}\\
-\mathbf{C}    & -\mathbf{G}   & 0_{n \times n}\\
0_{n \times n} & \mathbf{D}& 
0_{n \times n}\end{bmatrix} 
\vec{z}(t)
+ \begin{bmatrix}
        0 \\ \mathbf{F} \\ 0
    \end{bmatrix}
\end{aligned}
\end{equation}
for $\vec{z}(t) = (\mathfrak{a}_1(t), \ldots, \mathfrak{a}_n(t), \dot{\mathfrak{a}}_1(t), \ldots, \dot{\mathfrak{a}}_n(t),
\mathfrak{c}_1(t), \ldots, \mathfrak{c}_n(t))^T$, with periodicity conditions $\vec{z}(0)=\vec{z}(T)$.
Here
\begin{align*}
\mathbf{F}_j =& -\left(\tfrac{1}{\alpha} f,\Delta\phi_j\right)_{L^2(\Omega)} \\ 
     \mathbf{H}_{i,j} =& \left(\nabla \phi_i, \nabla \phi_j \right)_{L^2(\Omega)} \\
\mathbf{C}_{i,j} =& \left(\tfrac{1}{\alpha}c^2 \Delta \phi_i, \Delta \phi_j \right)_{L^2(\Omega)}, \\
\mathbf{G}_{i,j} =& 
\frac{\gamma}{\beta}\left( \nabla \phi_i, \nabla \phi_j \right)_{L^2(\Omega)} 
+\left(\tfrac{1}{\alpha}(b \Delta \phi_i+(\tfrac{\gamma}{\beta}\alpha-\mu)\phi_i), \Delta \phi_j \right)_{L^2(\Omega)}
\\&
+ \frac{1}{\beta}\left( \trace{\Omega}{(\nabla \phi_i\cdot\vecn)}, \trace{\Omega}{(\nabla \phi_j\cdot\vecn)} \right)_{L^2(\partial\Omega)}
\\
     \mathbf{D}_{i,j} =&   
\left( \eta_i, \eta_j \right)_{L^2(\partial\Omega)}
\end{align*}
Here, the matrices $\mathbf{H}$, $\mathbf{D}$ are positive definite. 

We can thus restate the problem as follows
\begin{equation}
    \begin{aligned}
    \label{eq:westervelt:linear:galerkin:form:full:ODE:system:b}
       \dot{\vec{z}}(t) &=  
\begin{bmatrix} 
I_{n \times n} & 0_{n \times n}& 0_{n \times n}\\ 
0_{n \times n} & \mathbf{H}     &0_{n \times n}\\
0_{n \times n} & 0_{n \times n}& \mathbf{D}\\ 
\end{bmatrix}^{-1} \left( \begin{bmatrix} 
0_{n \times n} & I_{n \times n}& 0_{n \times n}\\
-\mathbf{C}    & -\mathbf{G}   & 0_{n \times n}\\
0_{n \times n} & \mathbf{D}& 
0_{n \times n}\end{bmatrix} \vec{z}(t) 
+ \begin{bmatrix}
        0  \\ \mathbf{F} \\ 0
    \end{bmatrix} \right) \\
    &= \Tilde{\mathbf{A}}(t) \vec{z}(t) + \Tilde{\mathbf{F}}(t),
\end{aligned}
\end{equation}
where $\Tilde{\mathbf{A}}(0) = \Tilde{\mathbf{A}}(T)$. 
It is readily checked that the conditions for the Floquet-Lyapunov Theorem (see~\cite[p.~90]{linODEperiodic}) are fulfilled, since the system with vanishing $f$ only has the zero solution (see the energy estimates below). Therefore, we obtain the existence of a periodic solution in $C^2(0,T;V_n)$ to~\eqref{eq:westervelt:linear:galerkin:form:full:ODE:system:b}. 
Note that \eqref{eq:westervelt:linear:galerkin:form:full:ODE:system:a}  is equivalent to the spatially discretized version of \eqref{eq:westervelt:linear:weak:form} after reverting the partial integration with respect to time in the first term and setting $\psi(s)=\chi_{[0,t]}(s)\phi_j$. 

\smallskip

Third, we derive energy estimates providing the boundedness of the sequence of Galerkin solutions. 
We do so by 
testing the spatially discretized version of \eqref{eq:westervelt:linear:weak:form} with $\psi=\frac{\gamma}{2\beta}u_n+u_{n\,t}$, $\zeta=h_{n\,t}-u_{n\,t}-[\gamma+\vecn\cdot\nabla]u_n$. Using the periodicty of $\vec{z}$ (that is, of $u_n$ and $u_{n\,t}$) yields the energy identity
\begin{align*}
&\text{lhs}:=\frac{\gamma}{2\beta}\|\nabla u_{n\,t}\|_{L^2(0,T;L^2(\Omega))}^2
+\|\sqrt{\tfrac{b}{\alpha}}\Delta u_{n\,t}\|_{L^2(0,T;L^2\Omega))}^2
+\frac{\gamma}{2\beta} \|\tfrac{c}{\sqrt{\alpha}}\Delta u_n\|_{L^2(0,T;L^2(\Omega))}^2\\
&\qquad+\frac{1}{\beta} \|\nabla u_{n\,t}\cdot\vecn\|_{L^2(0,T;L^2(\partial\Omega))}^2
+ \|h_{n\,t}-u_{n\,t}-[\gamma+\vecn\cdot\nabla]u_n\|_{L^2(0,T;L^2(\partial\Omega))}^2\\
&=\int_0^T\int_\Omega\Bigl(
(\tfrac{1}{2\alpha})_t (c^2+\tfrac{\gamma}{2\beta}b)(\Delta u_n)^2 
-\tfrac{1}{\alpha}(f+(\tfrac{\gamma}{\beta}\alpha-\mu)u_{n\,t}) (\Delta u_{n\,t} + \tfrac{\gamma}{2\beta}\Delta u_n )
\Bigr)\, dx\, dt \\
&=:\text{rhs}.
\end{align*}
Here the right hand side can be further estimated by means of the Holder's and Young's inequality
\begin{align*}
|\text{rhs}|\leq&
\|(\tfrac{1}{2\alpha})_t (c^2+\tfrac{\gamma}{2\beta}b)
\|_{L^\infty(0,T;L^\infty(\Omega))}
\|\Delta u_n\|_{L^2(0,T;L^2(\Omega))}^2\\
&+\frac{\varepsilon_1+\varepsilon_2}{2}\|\Delta u_{n\,t} + \tfrac{\gamma}{2\beta}\Delta u_n\|_{L^2(0,T;L^2(\Omega))}^2
\\
& + \frac{1}{2\varepsilon_1}\|\tfrac{\gamma}{\beta}-\tfrac{\mu}{\alpha}\|_{L^\infty(0,T;L^{2q/(q-1)}(\Omega))}
\|u_{n\,t}\|_{L^2(0,T;L^{2q}(\Omega))}^2
+\frac{1}{2\varepsilon_2}\|\tfrac{f}{\alpha}\|_{L^2(0,T;L^2(\Omega))}^2
\end{align*}
Due to our ansatz space setting and elliptic regularity, skipping the $h_{n\,t}$ term, the left hand side can be bounded from below by 
\begin{align*}
    \text{lhs}&\geq
    \frac{\gamma}{2\beta}\|\nabla u_{n\,t}\|_{L^2(0,T;L^2(\Omega))}^2
+\|\sqrt{\tfrac{b}{\alpha}}\Delta u_{n\,t}\|_{L^2(0,T;L^2(\Omega))}^2
+\frac{\gamma}{2\beta} \|\tfrac{c}{\sqrt{\alpha}}\Delta u_n\|_{L^2(0,T;L^2(\Omega))}^2\\
&+\frac{\gamma}{2\beta} \|u_{n\,t}\|_{L^2(0,T;L^2(\partial\Omega))}^2
\\
    &\geq
    \frac{1}{C}\|u_{n\,t}\|_{L^2(0,T;H^{3/2}(\Omega))}^2
+\|\Delta u_n\|_{L^2(0,T;L^2(\Omega))}^2
\end{align*}
for some $C > 0$.
Together with continuity of the embedding $H^{3/2}(\Omega) \subseteq L^{2q}(\Omega)$, $q \geq 1$, and the imposed bounds on the coefficients, this yields an estimate of the form
\begin{equation}\label{enest0}
\|u_{n\,t}\|_{L^2(0,T;H^{3/2}(\Omega))}^2
+\|\Delta u_n\|_{L^2(0,T;L^2\Omega))}^2
\leq C \|f\|_{L^2(0,T;L^2(\Omega))}^2.
\end{equation}
On the other hand, we clearly also have 
\begin{align*}
    \text{lhs}&\geq \|h_{n\,t}-u_{n\,t}-[\gamma+\vecn\cdot\nabla]u_n\|_{L^2(0,T;L^2(\partial\Omega))}^2,
\end{align*}
from which together with the above estimate on $u_n$ we can can extract an estimate of the form 
\begin{equation}\label{enest0_h}
\|h_{n\,t}\|_{L^2(0,T;L^2(\partial\Omega))}^2
\leq C \|f\|_{L^2(0,T;L^2(\Omega))}^2.
\end{equation}

To also obtain an estimate on $u_{n\,tt}$, we test the spatially discretized version of \eqref{eq:westervelt:linear:weak:form} with $\bigl(\psi_n(t),\zeta_n(t)\bigr)=\bigl((-\Delta_\gamma)^{-1}u_{n\,tt}(t),0\bigr)\in V_n$, 
and 
\revision{integrate by parts to obtain (cf. \eqref{eq:westervelt:linear:weak:form_L2}) 
\[
\begin{aligned}
&\int_0^T\Bigl(\int_\Omega \frac{1}{\alpha}\Bigl(\alpha u_{n\,tt}-c^2\Delta u_n-b\Delta u_{n\,t}+\mu u_{n\,t}-f\Bigr)\, u_{n\,tt}\, dx + \int_{\partial\Omega} \frac{1}{\beta}(\beta u_{n\,tt}+\gamma u_{n\,t}+\nabla u_{n\,t}\cdot\vecn)\, (\nabla\psi_n\cdot\vecn)\, dS\Bigr) \,dt
= 0.
\end{aligned}
\]
We use the identity
\begin{align*}
&\int_0^T\int_{\partial\Omega} \frac{1}{\beta}(\beta u_{n\,tt}+\gamma u_{n\,t}+\nabla u_{n\,t}\cdot\vecn)\, (\nabla\psi_n\cdot\vecn)\, dS \,dt 
=-\int_0^T\int_{\partial\Omega} \frac{\gamma}{\beta}\Bigl(\beta u_{n\,tt}+\gamma u_{n\,t}+\nabla u_{n\,t}\cdot\vecn) \psi_n\, dS \,dt,
\end{align*}
that results from the fact that $\nabla\psi_n(t)\cdot\vecn=-\gamma\psi_n(t)$ holds by its definition as $\psi_n(t)=(-\Delta_\gamma)^{-1}u_{n\,tt}(t)$,
}
\MarginB{x.}
to arrive at
\begin{align*}
\|u_{n\,tt}\|_{L^2(0,T;L^2(\Omega))}^2
=&\left(\frac{1}{\alpha}\Bigl(c^2\Delta u_n+b\Delta u_{n\,t}-\mu u_{n\,t}+f\Bigr),\, u_{n\,tt}\right)_{L^2(0,T;L^2(\Omega))}\\
&\revision{-\frac{\gamma}{\beta}\int_0^T\langle \beta u_{n\,tt}+\gamma u_{n\,t}+\nabla u_{n\,t}\cdot\vecn,\psi_n\rangle_{H^{3/2}(\partial\Omega)^*,H^{3/2}(\partial\Omega)}}\, dt , 
\end{align*}
where due to the trace theorem and elliptic regularity 
\[
\|\psi_n(t)\|_{H^{3/2}(\partial\Omega)}\leq C_{\text{trace}} \|\Delta_\gamma\psi_n(t)\|_{L^2(\Omega)}
=C_{\text{trace}} \|u_{n\,tt}(t)\|_{L^2(\Omega)}.
\]
Applying the Cauchy-Schwarz inequality as well as \eqref{enest0} yields
\begin{equation}\label{enest1}
\begin{aligned}
&\|u_{n\,tt}\|_{L^2(0,T;L^2(\Omega))}^2+ \|u_{n\,t}\|_{L^2(0,T;H^{3/2}(\Omega))}^2
+\|\Delta u_n\|_{L^2(0,T;L^2(\Omega))}^2 \leq C \|f\|_{L^2(0,T;L^2(\Omega))}^2.
\end{aligned}
\end{equation}

\smallskip

Fourth, we take weak limits to construct a solution. 
The space $\Tilde{X} := \{ v \in H^2(0,T;L^2(\Omega)) \cap H^1(0,T;H^{3/2}(\Omega)) \, :\,  \Delta v\in L^2(0,T;L^2(\Omega)) \}$ induced by the energy estimates so far is a Hilbert space, thus reflexive according to the Riesz-Frechét representation theorem (see~\cite[p. 228]{Werner2011}). The estimates derived before state that $(u_n)_{n \in \mathbb{N}}$ is a bounded sequence in $\Tilde{X}$ and by the Eberlein-\v{S}mulian Theorem (see~\cite[p. 107]{Werner2011}), there exists a weakly convergent subsequence $(u_{n_k})_{k \in \mathbb{N}} \subseteq (u_n)_{n \in \mathbb{N}}$ such that $u_{n_k} \rightharpoonup u \in \Tilde{X}$ for $k \rightarrow \infty$. 
Due to linearity of the problem it is readily checked that $u$ indeed satisfies \eqref{eq:westervelt:linear:weak:form} and by combining the above estimates and using weakly lower semicontinuity of the norm, we obtain 
\begin{equation}\label{enestCXtil}
    \|u\|_{\Tilde{X}} \leq C \|f\|_{L^2(0,T;L^2(\Omega))}.
\end{equation}
In fact $u$ solves~\eqref{eq:westervelt:linear:periodic} in an $L^2(0,T;L^2(\Omega))$ sense via \eqref{eq:westervelt:linear:weak:form_L2} .

Fifth, we derive higher spatial regularity. Our goal is to show that a solution to~\eqref{eq:westervelt:linear:periodic} is contained in $X$. Therefore, we want to apply elliptic regularity theory (see~\cite[p. 326]{evans2010}). 
To this end, besides using the $L^2$ estimate on $\Delta u$ resulting from \eqref{enestCXtil}, we have to take into account the regularity of the boundary values. Using the boundary conditions and again estimate \eqref{enestCXtil}, we have that $[\gamma+\vecn\cdot\nabla] u(t) =-\beta u_t$ is the trace of an $H^1(\Omega)$ function and by the trace theorem (see~\cite[p. 272]{evans2010}) and the estimates above we further obtain 
\begin{align}\label{estTraceu}
    \|\trace{\Omega}{[\gamma+\vecn\cdot\nabla]u}\|_{L^2(0,T,H^{1/2}(\partial\Omega))}
    = \|\trace{\Omega}{\beta u_{t}}\|_{L^2(0,T,H^{1/2}(\partial\Omega))} \leq \Tilde{C} \|f\|_{L^2(0,T;L^2(\Omega))}.
\end{align}
Applying elliptic regularity for the impedance Laplacian yields the estimate
\begin{align}
    \|u\|_{L^2(0,T;H^2(\Omega))} \leq \Tilde{C} \|f\|_{L^2(0,T;L^2(\Omega))}.
\end{align}
This concludes the proof.

\section{Proof of Theorem~\ref{thm:westervelt:nonlinear:solution:existence:uniqueness}}
\label{appendix:proof:thm:3:2}

We start by defining the solution operator $\mathcal{F}: X \rightarrow X$, $v \mapsto u$, such that $u$ is a solution to 
\begin{equation}
\label{rmk:eq:westervelt:nonlinear:periodic:fixed:point}
\begin{cases}
u_{tt} - c^2 \Delta u - b \Delta u_t = \eta (v^2)_{tt} + h & \text{in } (0,T) \times \Omega,\\
\beta u_t + \gamma u + \nabla u \cdot \vecn = 0 & \text{on } (0,T) \times \partial\Omega, \\
u(0) = u(T), \, u_t(0) = u_t(T) & \text{in } \Omega. \\
\end{cases}   
\end{equation}
We will employ a fixed point argument. 
To this end we will show that $\mathcal{F}$, restricted to a suitable ball $B_r(0)$, $r > 0$, is a self-mapping and a contraction. Since $X$ is a Banach space, we thus obtain uniqueness and existence of a solution $u$ to \eqref{eq:westervelt:nonlinear:periodic} by the Banach fixed-point theorem.

First, we investigate the well-definedness of the solution operator. Therefore, let $r>0$ and $v \in B_r(0)$ be arbitrary. For the right hand side of~\eqref{rmk:eq:westervelt:nonlinear:periodic:fixed:point}, given the assumptions on $h$ and \revision{$\eta$}, we obtain the following estimate:
\begin{align}
    \nonumber
    & || \eta  (v^2)_{tt} + h ||_{L^2(0,T;L^2(\Omega))}  \leq 2 || \eta ||_{L^\infty(\Omega)} || v_t^2 +  v_{tt} v||_{L^2(0,T;L^2(\Omega))} + || h ||_{L^2(0,T;L^2(\Omega))}.
\end{align}
Hence, in order to apply theorem~\ref{thm:westervelt:linear:solution:existence:uniqueness} we need to establish $|| v_t^2 +  v_{tt} v||_{L^2(0,T;L^2(\Omega))} < \infty$. To do so, we exploit the Sobolev embedding theorem  (see,~\cite{evans2010} and~\cite{FAN2001749}) and (real) interpolation~\cite{Amann95} between Banach spaces which gives 
\begin{align}
    \label{eq:interpol:3:4}
    &H^{3/4}(0,T;H^{13/8}(\Omega)) = H^{3/4}(0,T;[H^2(\Omega), H^{3/2}(\Omega)]_{3/4}) = [L^2(0,T;H^2(\Omega)), H^1(0,T;H^{3/2}(\Omega))]_{3/4},
\end{align} 
and 
\begin{align}
    \label{eq:interpol:1:4}
    &H^{1/4}(0,T;H^{9/8}(\Omega)) = H^{1/4}(0,T;[H^{3/2}(\Omega), L^{2}(\Omega)]_{1/4}) = [L^2(0,T;H^{3/2}(\Omega)), H^1(0,T;L^{2}(\Omega))]_{1/4}.
\end{align}
With this, we proceed by estimating $||v||_{L^\infty(0,T;L^\infty(\Omega))}$ and $||v_t^2||_{L^2(0,T;L^2(\Omega))}$. Exploiting the Sobolev embedding theorem in space~\eqref{sobolev:space:1}, in time~\eqref{sobolev:space:time:1} and (real) interpolation~\eqref{eq:interpol:3:4} gives
\begin{align}
\nonumber
||v||_{L^\infty(0,T;L^\infty(\Omega))} &\leq C_{H^{13/8}(\Omega) \rightarrow L^\infty(\Omega)} ||v||_{L^\infty(0,T;H^{13/8}(\Omega))} \\ \nonumber
& \leq  C_{H^{13/8}(\Omega) \rightarrow L^\infty(\Omega)}  C_{H^{3/4}(0,T) \rightarrow L^\infty(0,T)}||v||_{H^{3/4}(0,T;H^{13/8}(\Omega))} \\ \label{eq:estimate:v:infty}
& \leq  C_{H^{13/8}(\Omega) \rightarrow L^\infty(\Omega)}  C_{H^{3/4}(0,T) \rightarrow L^\infty(0,T)}  ||v||_{H^{1}(0,T;H^{3/2}(\Omega))}^{3/4} ||v||_{L^{2}(0,T;H^{2}(\Omega))}^{1/4}
 \\ \nonumber
& \leq  C_{H^{13/8}(\Omega) \rightarrow L^\infty(\Omega)}  C_{H^{3/4}(0,T) \rightarrow L^\infty(0,T)} ||v||_X \\ \nonumber
&\leq C_{H^{13/8}(\Omega) \rightarrow L^\infty(\Omega)}  C_{H^{3/4}(0,T) \rightarrow L^\infty(0,T)} r.
\end{align}
Hence, $v \in L^\infty(0,T;L^\infty(\Omega))$ as required. For $||v_t^2||_{L^2(0,T;L^2(\Omega))}$ we proceed similarly, using~\eqref{sobolev:space:2}, \eqref{sobolev:space:time:2} and the interpolation identity in~\eqref{eq:interpol:3:4} which results in the following estimate 
\begin{align}
    \nonumber
    || v_t^2 ||_{L^2(0,T;L^2(\Omega))} & = || v_t||_{L^4(0,T;L^4(\Omega))}^2 \\ \nonumber
    & \leq  C_{H^{1/4}(0,T) \rightarrow L^4(0,T)}^2 C_{H^{9/8}(\Omega) \rightarrow L^4(\Omega)}^2 ||v_t||_{H^{1/4}(0,T;H^{9/8}(\Omega))}^2 \\ \label{eq:estimate:v:t:2}
    & \leq C_{H^{1/4}(0,T) \rightarrow L^4(0,T)}^2 C_{H^{9/8}(\Omega) \rightarrow L^4(\Omega)}^2   \left(||v_t||_{H^1(0,T;L^2(\Omega))}^{1/4} ||v_t||_{L^2(0,T;H^{3/2}(\Omega))}^{3/4} \right)^2  \\ \nonumber
    & \leq C_{H^{1/4}(0,T) \rightarrow L^4(0,T)}^2 C_{H^{9/8}(\Omega) \rightarrow L^4(\Omega)}^2 ||v||_X^2 \\ \nonumber
    & \leq  C_{H^{1/4}(0,T) \rightarrow L^4(0,T)}^2 C_{H^{9/8}(\Omega) \rightarrow L^4(\Omega)}^2 r^2.
\end{align}
The estimates in~\eqref{eq:estimate:v:infty} and~\eqref{eq:estimate:v:t:2} imply that $f = \eta(v^2)_{tt} + h \in L^2(0,T;L^2(\Omega))$, in particular the estimate with respect to $r > 0$ reads
\begin{align}
    \nonumber
    & 2 || \eta ||_{L^\infty(\Omega)} || v_t^2 +  v_{tt} v||_{L^2(0,T;L^2(\Omega))} + || h ||_{L^2(0,T;L^2(\Omega))} \leq ||h||_{L^2(0,T;L^2(\Omega))}  + 4 ||\eta||_{L^\infty(\Omega)} \\ \nonumber
    & r^2 \left(C_{H^{1/4}(0,T) \rightarrow L^4(0,T)}^2 C_{H^{9/8}(\Omega) \rightarrow L^4(\Omega)}^2 + C_{H^{13/8}(\Omega) \rightarrow L^\infty(\Omega)}  C_{H^{3/4}(0,T) \rightarrow L^\infty(0,T)} \right).
\end{align}
Theorem~\ref{thm:westervelt:linear:solution:existence:uniqueness}, with $\alpha = 1$ and $\mu = 0$, implies the existence of a constant $C > 0$ independent of $f, \eta$ and $v$ such that
\begin{equation}
    \nonumber
    ||\mathcal{F}(v)||_X \leq C||\eta(v^2)_{tt} + h||_{L^2(0,T;L^2(\Omega))}.
\end{equation}
Hence, for $v \in X$, $\mathcal{F}(v)$ is indeed a (weak) solution to~\eqref{rmk:eq:westervelt:nonlinear:periodic:fixed:point}. The established estimates provide a hint what conditions we have to require for $r>0$ and $\delta>0$ to ensure that $\mathcal{F}: X \rightarrow X$ becomes a self-mapping and a contraction on $B_r(0)$. These conditions are as follows:
\begin{align}
        \label{eq:westervelt:nonlinear:condition:3:modified}
         C(4||\eta||_{L^\infty(\Omega)}(C_{H^{1/4}(0,T) \rightarrow L^4(0,T)}^2 C_{H^{9/8}(\Omega) \rightarrow L^4(\Omega)}^2 + C_{H^{13/8}(\Omega) \rightarrow L^\infty(\Omega)}  C_{H^{3/4}(0,T) \rightarrow L^\infty(0,T)}) r^2 + \delta ) < r,
\end{align}
and
\begin{align}
        \label{eq:westervelt:nonlinear:condition:4}
        ||\eta||_{L^\infty(\Omega)} C_{H^{13/8}(\Omega) \rightarrow L^\infty(\Omega)}  C_{H^{3/4}(0,T) \rightarrow L^\infty(0,T)} r < \tfrac{1}{2}.
\end{align}
Condition~\eqref{eq:westervelt:nonlinear:condition:3:modified} ensures that $\mathcal{F}(B_r(0)) \subseteq B_r(0)$. Solving~\eqref{eq:westervelt:nonlinear:condition:3:modified} with respect to $\delta$ gives
\begin{align}
    \nonumber
    \delta < r \left(   \frac{1}{C} - 4||\eta||_{L^\infty(\Omega)}(C_{H^{1/4}(0,T) \rightarrow L^4(0,T)}^2 C_{H^{9/8}(\Omega) \rightarrow L^4(\Omega)}^2 + C_{H^{13/8}(\Omega) \rightarrow L^\infty(\Omega)}  C_{H^{3/4}(0,T) \rightarrow L^\infty(0,T)}) r \right).
\end{align}
Thus, $r>0$ must be sufficiently small to satisfy
\begin{align}
    \nonumber
    \frac{1}{C} > 4||\eta||_{L^\infty(\Omega)}(C_{H^{1/4}(0,T) \rightarrow L^4(0,T)}^2 C_{H^{9/8}(\Omega) \rightarrow L^4(\Omega)}^2 + C_{H^{13/8}(\Omega) \rightarrow L^\infty(\Omega)}  C_{H^{3/4}(0,T) \rightarrow L^\infty(0,T)}) r.
\end{align}
Condition~\eqref{eq:westervelt:nonlinear:condition:4} ensures that degeneracy is avoided, as follows: 
\begin{align}
\nonumber
& \sup\{c \in \mathbb{R}: 1 - 2 \eta(x) u(t,x) \geq c,\text{ a.e. in }  (0,T) \times \Omega \} = 1 - \inf\{c \in \mathbb{R}: 2 \eta(x) u(t,x) \leq c,\text{ a.e. in }  (0,T) \times \Omega \} \\ \nonumber
&\geq 1 - 2 ||\eta u||_{L^\infty(0,T;L^\infty(\Omega))} \geq 1 - 2 ||\eta||_{L^\infty(\Omega)}  C_{H^{13/8}(\Omega) \rightarrow L^\infty(\Omega)}  C_{H^{3/4}(0,T) \rightarrow L^\infty(0,T)} r > 0.
\end{align}
Next, we verify that $\mathcal{F}$ is a contraction. Let $w = u_1 - u_2  = \mathcal{F}(v_1) - \mathcal{F}(v_2)$ for $v_1, v_2 \in B_r(0)$. Then $w$ solves
\begin{equation}
\label{eq:westervelt:nonlinear:periodic:fixed:point:quadratic}
\begin{cases}
w_{tt} - c^2 \Delta w - b \Delta w_t = \eta \left((v_1^2)_{tt} - (v_2^2)_{tt} \right)  & \text{in } (0,T) \times \Omega,\\
\beta w_t + \gamma w + \nabla w \cdot \vecn = 0 & \text{on } (0,T) \times \partial\Omega, \\
w(0) = w(T), \, w_t(0) = w_t(T) & \text{in } \Omega. \\
\end{cases}   
\end{equation}
By theorem~\ref{thm:westervelt:linear:solution:existence:uniqueness}, we obtain the estimate
\begin{equation}
    \nonumber
    ||w||_X \leq C ||\eta \left((v_1^2)_{tt} - (v_2^2)_{tt} \right)||_{L^2(0,T;L^2(\Omega))}.
\end{equation}
Reorganizing $\left((v_1^2)_{tt} - (v_2^2)_{tt} \right)$ yields
\begin{align}
    (v_1^2)_{tt} - (v_2^2)_{tt} &= 2 \left( (v_{1_t} + v_{2_t})(v_{1_t} - v_{2_t}) + v_1(v_{1_{tt}} -  v_{2_{tt}}) + v_{2_{tt}} (v_1 - v_2) \right).
\end{align}
Applying Minkowski's and H\"older's inequalities, we have
\begin{align}
    \nonumber
    & || (v_{1_t} + v_{2_t})(v_{1_t} - v_{2_t}) + v_1(v_{1_{tt}} -  v_{2_{tt}}) + v_{2_{tt}} (v_1 - v_2) ||_{L^2(0,T;L^2(\Omega))} \\ \nonumber
    & \leq ||v_{1_t} + v_{2_t}||_{L^4(0,T;L^4(\Omega))}  ||(v_{1} - v_{2})_t||_{L^4(0,T;L^4(\Omega))}  \\ \nonumber
    & + ||v_1||_{L^\infty(0,T;L^\infty(\Omega))} ||(v_{1} -  v_{2})_{tt}||_{L^2(0,T;L^2(\Omega))} \\ \nonumber
    &+ ||v_1 - v_2||_{L^\infty(0,T;L^\infty(\Omega))} ||v_{2_{tt}}||_{L^2(0,T;L^2(\Omega))}.
\end{align}
From the established estimates in~\eqref{eq:estimate:v:infty} and~\eqref{eq:estimate:v:t:2} we conclude 
\begin{align}
    ||\mathcal{F}(v_1) - \mathcal{F}(v_2)||_X \leq rC_0 ||v_1 - v_2||_X,
\end{align}
where
\begin{align}
    C_0 := 4 C ||\eta||_{L^\infty(\Omega)} (C_{H^{13/8}(\Omega) \rightarrow L^\infty(\Omega)}  C_{H^{3/4}(0,T) \rightarrow L^\infty(0,T)} + C_{H^{1/4}(0,T) \rightarrow L^4(0,T)}^2 C_{H^{9/8}(\Omega) \rightarrow L^4(\Omega)}^2).
\end{align}
The imposed condition~\eqref{eq:westervelt:nonlinear:condition:3:modified} ensures $rC_0 < 1$, and dividing by $r>0$ gives 
\begin{align}
        \nonumber
         rC_0 & = rC(4||\eta||_{L^\infty(\Omega)}(C_{H^{1/4}(0,T) \rightarrow L^4(0,T)}^2 C_{H^{9/8}(\Omega) \rightarrow L^4(\Omega)}^2 + C_{H^{13/8}(\Omega) \rightarrow L^\infty(\Omega)}  C_{H^{3/4}(0,T) \rightarrow L^\infty(0,T)})) \\
         & \leq  rC(4||\eta||_{L^\infty(\Omega)}(C_{H^{1/4}(0,T) \rightarrow L^4(0,T)}^2 C_{H^{9/8}(\Omega) \rightarrow L^4(\Omega)}^2  + C_{H^{13/8}(\Omega) \rightarrow L^\infty(\Omega)}  C_{H^{3/4}(0,T) \rightarrow L^\infty(0,T)})) + C \frac{\delta}{r} < 1.
\end{align}
Thus, $\mathcal{F}$ is a contraction on $B_r(0)$ and, therefore has a unique fixed point $u \in B_r(0)$ with $\mathcal{F}(u) = u$ solving~\eqref{eq:westervelt:nonlinear:periodic}.

\section{Proof of Proposition~\ref{prop:helmholtz:fully:coupled}}
\label{appendix:proof:prop:4:1}
Let $u^N, v^N \in X_N$ then
\begin{align}
    u^N v^N & = \frac{1}{2} \left(\sum_{m=0}^N \hat{u}_m(x) e^{\iota m \omega t} + \sum_{m=0}^N \overline{\hat{u}_m(x)} e^{- \iota m \omega t} \right) \frac{1}{2} \left(\sum_{m=0}^N \hat{v}_m(x) e^{\iota m \omega t} + \sum_{m=0}^N \overline{\hat{v}_k(x)} e^{- \iota m \omega t} \right) \\ \nonumber 
    & = \frac{1}{2} \text{Re} \left[ \sum_{m=0}^{2N} \sum_{j = \max\{m-N,0\}}^{\min\{m,N\}}  \hat{u}_j(x) \hat{v}_{m-j}(x) e^{\iota m \omega t} + \overline{\hat{u}_j(x)} \hat{v}_{m-j}(x) e^{\iota (m - 2j) \omega t} \right],
\end{align}
\MarginB{xi.}
where we used the formula for the finite Cauchy product. A thorough computation of the second sum yields
\begin{align}
    &\text{Re} \left[\revision{\sum_{m=0}^{2N}} \sum_{j = \max\{m-N,0\}}^{\min\{m,N\}} \overline{\hat{u}_j(x)} \hat{v}_{m-j}(x) e^{\iota (m - 2j) \omega t} \right] \\ \nonumber
    & = \real{ \sum_{j=0}^N \overline{\hat{u}_j(x)} \hat{v}_{j}(x) + \sum_{m=1}^N e^{\iota m \omega t} \sum_{k=m:2}^{2N-m}\left[ \overline{\hat{u}_{\frac{k-m}{2}}(x)} \hat{v}_{\frac{k+m}{2}}(x) + \hat{u}_{\frac{k+m}{2}}(x) \overline{\hat{v}_{\frac{k-m}{2}}(x)}\right]}.
\end{align}
After projection, the first sum becomes (since we skip all harmonic frequencies above the $N$-th harmonic)
\begin{align}
    \nonumber
    \text{Proj}_{X_N} \left(\real{\sum_{m=0}^{2N} \sum_{j = \max\{m-N,0\}}^{\min\{m,N\}}  \hat{u}_j(x) \hat{v}_{m-j}(x) e^{\iota m \omega t} } \right) = \real{ \sum_{m=0}^{N} e^{\iota m \omega t} \sum_{j = 0}^{m}  \hat{u}_j(x) \hat{v}_{m-j}(x) }.
\end{align}
This finally yields
\begin{align}
\nonumber
    &\text{Proj}_{X_N}(u^N v^N) = \frac{1}{2}\real{ \sum_{m=0}^{N} e^{\iota m \omega t} \sum_{j = 0}^{m}  \hat{u}_j(x) \hat{v}_{m-j}(x)} \\ \nonumber
    &+  \frac{1}{2} \real{ \sum_{j=0}^N \overline{\hat{u}_j(x)} \hat{v}_{j}(x) + \sum_{m=1}^N e^{\iota m \omega t} \sum_{k=m:2}^{2N-m}\left[ \overline{\hat{u}_{\frac{k-m}{2}}(x)} \hat{v}_{\frac{k+m}{2}}(x) + \hat{u}_{\frac{k+m}{2}}(x) \overline{\hat{v}_{\frac{k-m}{2}}(x)}\right] }.
\end{align}
For $\text{Proj}_{X_N} \left( ((p^N)^2)_{tt} \right)$ we obtain the following expression
\begin{align}
\label{eq:westervelt:projection:qudratic:tt}
    \text{Proj}_{X_N} \left( ((p^N)^2)_{tt}  \right)  = -\frac{1}{2} &\omega^2 \text{Re}\left[ \sum_{m=1}^N m^2 e^{\iota m \omega t} \left(\sum_{k=1}^m \hat{p}_{k}(x) \hat{p}_{m-k}(x) + 2\sum_{k=m:2}^{2N-m} \overline{\hat{p}_{\frac{k-m}{2}}(x)} \hat{p}_{\frac{k+m}{2}}(x) \right)  \right].
\end{align}
%
For the left hand side of~\eqref{eq:westervelt:quadratic-nonlinearity:periodic} we obtain 
\begin{align}
    \nonumber
    \text{Proj}_{X_N}\left(p^N_{tt} - c^2 \Delta p^N - b \Delta p^N_t \right) &= \omega^2 \frac{1}{2} \sum_{k=0}^N - \hat{p}_k(x) k^2 e^{\iota k \omega t} - \overline{\hat{p}_k(x)} k^2 e^{-\iota k \omega t} \\ \nonumber
    & - c^2 \frac{1}{2} \sum_{k=0}^N \Delta \hat{p}_k(x) e^{\iota k \omega t} + \Delta \overline{\hat{p}_k(x)} e^{-\iota k \omega t} \\ \nonumber 
    & - \iota \omega b \frac{1}{2} \sum_{k=0}^N  \Delta\hat{p}_k(x) k e^{\iota k \omega t} - \Delta\overline{\hat{p}_k(x)} k e^{-\iota k \omega t}.
\end{align}
For the source term we further have $h(t,x) = g(t,x)_{tt}$ which gives
\begin{equation}
    \nonumber
    \text{Proj}_{X_N}(h(x,t)) = -\frac{\omega^2}{2} \sum_{k=1}^N k^2 \hat{h}_k e^{i k \omega t} + k^2 \overline{\hat{h}_k} e^{-i k \omega t}.
\end{equation}
Considering the linear independence of $e^{\iota m \omega t}$ in $L^2(0,T)$ we obtain for each $0\leq m \leq N$ that $\hat{p}_m(x)$ has to solve
\begin{align}
    \nonumber
    -\omega^2 m^2 & \hat{p}_m(x)  - ( c^2 + \iota m \omega b) \Delta \hat{p}_m(x) = \\ \nonumber
    & - \frac{\eta(x)\omega^2 m^2}{2} \left(\sum_{k=1}^m \hat{p}_{k}(x) \hat{p}_{m-k}(x) + 2\sum_{k=m:2}^{2N-m} \overline{\hat{p}_{\frac{k-m}{2}}(x)} \hat{p}_{\frac{k+m}{2}}(x)\right) - \omega^2 m^2 \hat{h}_m(x).
\end{align}
Dividing by $(c^2 + \iota m \omega b)$ gives
\begin{align}
    \nonumber
    -\kappa^2 m^2 & \hat{p}_m(x) - \Delta \hat{p}_m(x) = - \frac{\eta(x) \kappa^2 m^2}{2} \left(\sum_{k=1}^m \hat{p}_{k}(x) \hat{p}_{m-k}(x) + 2\sum_{k=m:2}^{2N-m} \overline{\hat{p}_{\frac{k-m}{2}}(x)} \hat{p}_{\frac{k+m}{2}}(x)\right) - \kappa^2 m^2 \hat{h}_m(x), 
\end{align}
$ m = 0$ results in
\begin{align}
    - \Delta  \hat{p}_0(x) = 0,
\end{align}
and the boundary conditions for each of the Helmholtz equations read
\begin{align}
    (\iota \beta m \omega + \gamma) \hat{p}_m(x) + \nabla \hat{p}_m(x) \cdot \vecn = 0,
\end{align}
where we define $\kappa^2 = \frac{\omega^2}{c^2 + \iota m \omega b}$. This concludes the proof.

%

\end{document}